\documentclass[reqno]{amsart}
\usepackage{amsmath,amssymb}
\usepackage{hyperref}
\usepackage{graphicx}

\newtheorem{thm}{Theorem}[section]

\newtheorem{lem}[thm]{Lemma}

\newtheorem{defn}[thm]{Definition}
\newtheorem{rem}[thm]{Remark}

\numberwithin{equation}{section}

\makeatletter
\newcommand{\arxiv}[1]{\href{http://arxiv.org/abs/#1}{arXiv:#1}}
\newcommand*{\mailto}[1]{\href{mailto:#1}{\nolinkurl{#1}}}
\newcommand{\rmnum}[1]{\romannumeral #1}
\newcommand{\Rmnum}[1]{\expandafter\@slowromancap\romannumeral #1@}
\newcommand{\R}{{\mathbb R}}
\newcommand{\Z}{{\mathbb Z}}
\newcommand{\N}{{\mathbb N}}
\newcommand{\K}{{\mathbb K}}
\makeatother

\def\qq{\qquad}

\def\x#1{(\ref{#1})}
\def\xx#1{{\rm (\ref{#1})}}

\def\disp{\displaystyle}
\def\rd{{\rm d}}

\begin{document}
\title{Gap Labelling for Almost Periodic Sturm-Liouville Operators}
	
\author[G. Teschl]{Gerald Teschl}
\address{Faculty of Mathematics\\ University of Vienna\\
Oskar-Morgenstern-Platz 1\\ 1090 Wien\\ Austria}
\email{\mailto{Gerald.Teschl@univie.ac.at}}
\urladdr{\url{https://www.mat.univie.ac.at/~gerald/}}

\author[Y. Wang]{Yifei Wang}
\address{State Key Laboratory of Mathematical Sciences\\ Academy of Mathematics and Systems Science\\  Chinese Academy of Sciences\\ 100190  Beijing\\ China\\ School of Mathematical Sciences\\ University of Chinese Academy of Sciences\\ 100049 Beijing\\ China}
\email{\mailto{wangyifei@amss.ac.cn}}

\author[B. Xie]{Bing Xie}
\address{School of Mathematics and Statistics\\ Shandong University\\ 264209 Weihai\\ China}
\email{\mailto{xiebing@sdu.edu.cn}}

\author[Z. Zhou]{Zhe Zhou}
\address{State Key Laboratory of Mathematical Sciences\\ Academy of Mathematics and Systems Science\\  Chinese Academy of Sciences\\ 100190  Beijing \\ China}
\email{\mailto{zzhou@amss.ac.cn}}

\begin{abstract}	
In this paper, we introduce a rotation number for almost periodic Sturm-Liouville operators in the spirit of Johnson and Moser. We then prove the gap labelling theorem in terms of rotation numbers for the operator in question. To do this, we rigorously prove the almost periodicity of Green's functions.
\end{abstract}
	
\maketitle
	
\tableofcontents

\section{Introduction}

\subsection{Background}

In their landmark paper \cite{JM1982}, Johnson and Moser introduced the concept of a rotation number for Schr\"odinger operators with almost periodic potentials. A brief description is as follows. For $\lambda\in\mathbb{R}$, let $\phi$ be a solution of the differential equation 
	\[
	H\phi:=-\frac{\mathrm{d}^2\phi}{\mathrm{d}x^2}+q(x)\phi=\lambda\phi,
	\]
where $H$ denotes the Schr\"odinger operator with a Bohr almost periodic potential $q(x)$. The rotation number for $H$ is defined as the average winding per unit of the associated two-vector given by the solution $\phi$ and its derivative $\phi^\prime$ around the origin in the $(\phi^\prime, \phi)$-plane. They used the so-called \textsl{rotation number} to establish the gap labelling theorem and characterized the spectrum of the operator in question. 

Regarding the gap labelling theorem, Jean Bellissard has made fundamental contributions by establishing a topological framework for spectral gaps in a variety of settings, including Schr\"odinger operators \cite{B1982, BS1982, B1992, BBG1992}, automatic sequences \cite{B1992, B1993}, and tiling dynamical systems \cite{BBG2006, SB2009}. Bellissard showed that each spectral gap is naturally labelled by an element of the $K_0$-group of the associated $C^*$-algebra, thereby revealing spectral gaps as stable topological invariants rather than purely analytical features. For a comprehensive review, see also \cite{B1986}.

Recently, the concept of the rotation number has been extended to more general potentials, such as the Stepanov almost periodic functions, almost periodic functions with $\delta$-interactions, and the so-called $\alpha$-norm almost periodic measures; see \cite{Z2015, DZ2021, DMZZ2024}. However, the spectral analysis of the corresponding operators via the rotation number method has yet to be developed.

	
In this paper, we consider the almost periodic Sturm-Liouville operators as follows.
	\begin{equation}  \label{sl}
		\begin{split}
			L_{\frac{1}{p},q,w}:\quad \mathcal{D}(L_{\frac{1}{p},q,w})&\rightarrow \mathcal{L}^2(\mathbb{R},w(x)\mathrm{d}x)\\
			f:=f(x) &\mapsto \tau_{\frac{1}{p},q,w}f:= \frac{1}{w(x)}\left( -\frac{\mathrm{d}}{\mathrm{d}x}p(x)\frac{\mathrm{d}f(x)}{\mathrm{d}x}+q(x)f(x)\right) ,
		\end{split}
	\end{equation}
where $ \mathcal{D}(L_{\frac{1}{p},q,w}) \subset \mathcal{L}^2(\mathbb{R},w(x)\mathrm{d}x)$ is a suitable domain so that $L_{\frac{1}{p},q,w}$ is self-adjoint, and $p:=p(x), q:=q(x), w:=w(x)$ are Bohr almost periodic functions. For $\lambda\in\mathbb{R}$, let $\phi$ be a solution of the differential equation 
	\begin{equation} \label{sp-sl}
		\tau_{\frac{1}{p},q,w}\phi=\lambda\phi.
	\end{equation}
The rotation number is defined as the average winding per unit of the associated two-vector given by the solution $\phi$ and its quasi-derivative $p\phi^\prime$ around the origin in the $(p\phi^\prime, \phi)$-plane. The spectral analysis of the almost periodic Sturm-Liouville operators will be addressed in detail via the rotation number method.
	
A typical model of almost periodic Sturm-Liouville operators is the case where $p, q, w$ are periodic functions with the same period. Research on periodic Sturm-Liouville operators is extensive, such as \cite{BR2004, MYZ2010, BSTT2023} and references therein. However, unlike Schr\"odinger operators, the Sturm-Liouville operator involves three coefficients, which means it may not be periodic even if all three coefficients are periodic. This may have an impact on the corresponding spectrum. For example, consider the Sturm-Liouville operator with all coefficient functions having the same period $T>0$. Using the method in \cite{Z2001}, one can deduce that all periodic and anti-periodic eigenvalues agree with the endpoints of $\left\lbrace \lambda\in\mathbb{R}: \rho(\lambda)=\frac{k\pi}{T}, k\in\mathbb{Z}\right\rbrace $, where $\rho(\lambda)$ denotes the rotation number. Figure \ref{figp} shows an intuitive example with			 	
	$$
	p(x)=\frac{1}{\sin x + 2},\quad q(x)=2\cos x, \quad w(x)=-\cos x + 2,
	$$ 
having the same period $2\pi$. Platforms at $\lambda=\frac{k}{2}, k\in\mathbb{Z}$, are observable. Replace $q(x)$ with $2\cos (\sqrt{2}x)$. Then $p, q, w$ are still periodic functions but the period of $p$ and $w$ is rationally independent of that of $q$. As shown in Figure \ref{figap}, the platforms are no longer regularly distributed at $\frac{\mathbb{Z}}{2}$.
\begin{figure}[hbt!]
	\centering
	\includegraphics[width=.66\textwidth]{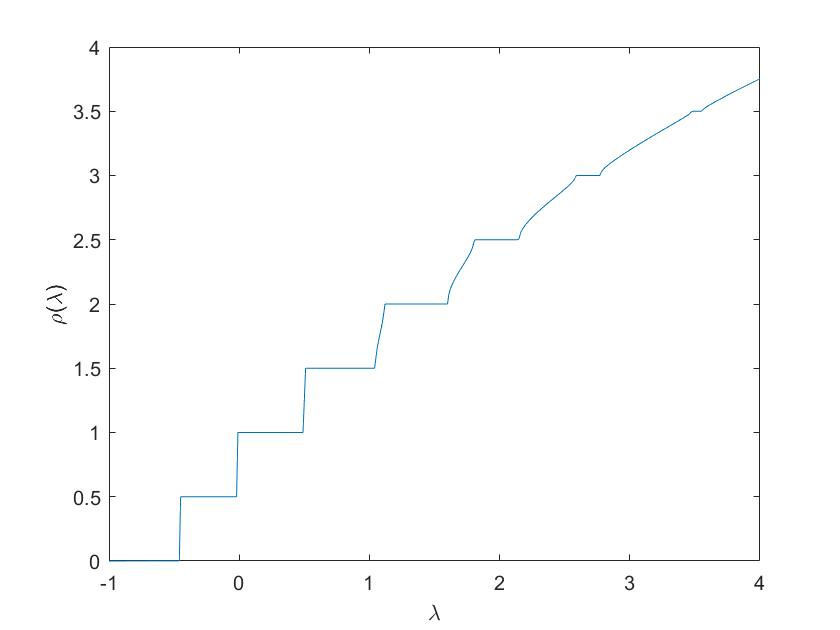}
	\caption{$p(x)=\frac{1}{\sin x + 2}$, $q(x)=2\cos x$, $w(x)=-\cos x + 2$} \label{figp}
\end{figure}
\begin{figure}[hbt!]
	\centering
	\includegraphics[width=.66\textwidth]{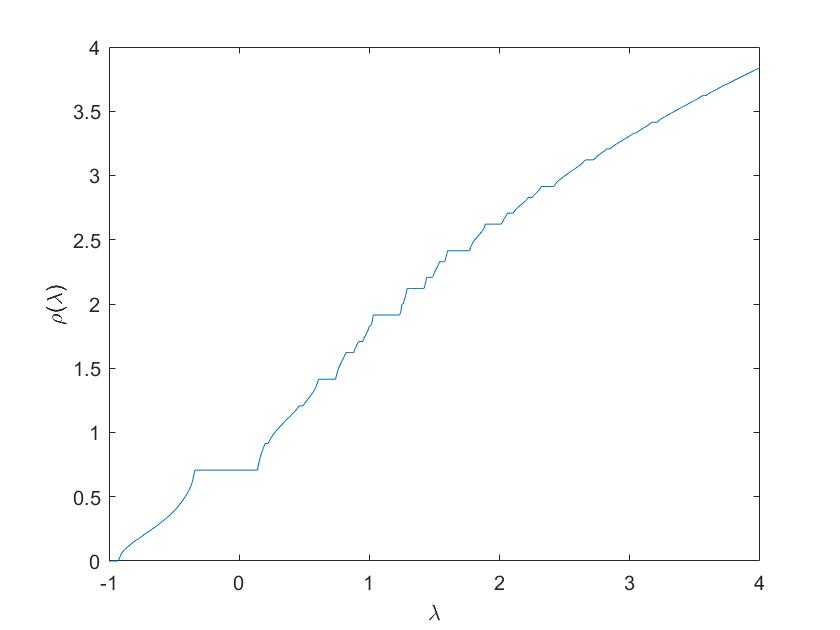}
	\caption{$p(x)=\frac{1}{\sin x + 2}$, $q(x)=2\cos (\sqrt{2}x)$, $w(x)=-\cos x + 2$} \label{figap}
\end{figure}

The Sturm-Liouville operator \x{sl} with
	$$
	p(x)=\frac{1}{\sin x + 2},\quad q(x)=2\cos (\sqrt{2}x), \quad w(x)=-\cos x + 2,
	$$  
is almost periodic rather than periodic. It is reasonable to expect that almost periodic Sturm-Liouville operators will have a wide range of applications, with their spectra exhibiting a richer array of phenomena. For instance, the spectrum of almost periodic Schr\"odinger operators displays remarkably complex behavior; see the survey by Simon \cite{Si2000} and the references therein.

Besides the paper \cite{JM1982} that we mentioned in the beginning, there are also quite a few papers to study the spectrum via the rotation number method. For instance, Giachett and Johnson \cite{GJ1984} studied two dimensional linear systems with bounded coefficients and proved the rotation number increases exactly on the spectrum. Sun \cite{S1991} studied one dimensional stationary ergodic Dirac operators on a probability space, and proved the relationship that the increasing points of the rotation number are exactly the spectrum for almost everywhere in the sample space. Fabbri, Johnson, and \'N\~unez \cite{FJN20031, FJN20032} studied non-autonomous linear Hamiltonian systems and showed that, under certain assumptions, the rotation number remains constant on open intervals where an exponential dichotomy exists. Consequently, the rotation number can be used to label the different gaps of the spectrum. For a comprehensive list of related works, see the book by Johnson, Obaya, Novo, \'N\~unez and Fabbri \cite{JONNF2016} and the references therein.

\subsection{The setting and main results}
	
Let us now describe the setting we consider in more details. Let $f: \mathbb{R}\rightarrow\mathbb{K}$ be a continuous function. Denote the shift of $f$ by
	\begin{equation} \label{sft-f}
		f\cdot t:=f(\cdot+t), \qquad \mbox{for any~}t \in \R.
	\end{equation}
We say that $f$ is \textsl{Bohr almost periodic} if one can extract a sub-sequence $\left\lbrace t_n\right\rbrace$ from any sequence $\left\lbrace \tilde{t}_n\right\rbrace\subseteq\mathbb{R}$ such that $\lim\limits_{n\rightarrow\infty}f\cdot t_n$ exists uniformly on the real axis. The space of all such functions is denoted by $\mathcal{AP}(\mathbb{R},\mathbb{K})$. It is well known that $\mathcal{AP}(\mathbb{R},\mathbb{K})$ is a Banach space with respect to the uniform norm $\|\cdot\|_{\infty}$; see \cite{Fink}. 
	
For any $f\in\mathcal{AP}(\mathbb{R},\mathbb{K})$, one can associate $f$ with the so-called \textsl{mean value} as
	\begin{equation*} 
		\mathrm{M}_x(f):= \lim\limits_{t\rightarrow+\infty}\frac{1}{t}\int_{0}^{t}f(x)\,\mathrm{d}x.
	\end{equation*}
The limit is well-defined by the definition of Bohr almost periodic functions. For $\lambda \in \mathbb{R}$, we denote the corresponding \textsl{Fourier coefficient} as
	\[
	\hat{f}(\lambda):=\mathrm{M}_x\big(f(x) \mathrm{e}^{-\mathrm{i}\lambda x}\big),
	\]
and define the set of \textsl{Fourier exponents} of $f$ by $\exp(f):=\{\lambda \in \mathbb{R}: \hat{f}(\lambda) \neq 0\}.$ The smallest additive group of $\R$ containing $\exp(f)$ is denoted by $\mathcal{M}_f$. We call it the \textsl{frequency module} of $f$. The \textsl{hull} of $f$ is defined by
	$$
	\mathrm{E}(f):=\overline{\left\lbrace f\cdot t: t\in\mathbb{R} \right\rbrace }^{\|\cdot\|_{\infty}}.
	$$ 
It is well known that $\mathrm{E}(f)$ is a compact and minimal set with the Haar measure, denoted by $\mu_{\mathrm{E}(f)}$, being the unique invariant measure of $\mathrm{E}(f)$ under the shift $f\mapsto f\cdot t$. For our purpose, we introduce the following space
	\begin{equation} \label{ap+}
		\mathcal{AP}_+(\mathbb{R},\mathbb{R}):=\big\{f \in \mathcal{AP}(\mathbb{R},\mathbb{R}): \text{any function}\; \widetilde{f}  \in \mathrm{E}(f) \; \text{is positive}\big\}.
	\end{equation}
Note that $\mathcal{AP}_+(\mathbb{R},\mathbb{R})$ is not a Banach space, and $\mathcal{AP}_+(\mathbb{R},\mathbb{R}) \neq \mathcal{AP}(\mathbb{R},\mathbb{R}_+)$ because of the example $\cos(x)+\cos(\sqrt{2}x)+2$.
	
From now on, we always assume that 
	$$
	p,w\in \mathcal{AP}_+(\mathbb{R},\mathbb{R}),\; \text{and} \; q \in\mathcal{AP}(\mathbb{R},\mathbb{R}), \eqno(*)
	$$
and consider the almost periodic Sturm-Liouville operator \x{sl}. It is natural to require that $p$ and $w$ are positive. However, this is not sufficient. When one studies almost periodic differential equations, the relevant equations in the hull also need to be considered. Like the Favard's theorem \cite[Theorem 6.3]{Fink}, when confirming the existence of almost periodic solutions, one must consider all the non-trivial solutions in the homogeneous hull of the original almost periodic system. That is why $\mathcal{AP}_+(\mathbb{R},\mathbb{R})$ was introduced in (\ref{ap+}). If $p$ and $w$ degenerate as periodic functions, then the assumption is equivalent to the positivity of $p$ and $w$ themselves.  

For the sake of simplicity, we take the following notation
	\begin{equation} \label{v}
		\textbf{v}:=\Big(\frac{1}{p},q,w\Big) \in \mathcal{AP}_*(\R,\R^3):= \mathcal{AP}_+(\mathbb{R},\mathbb{R}) \times \mathcal{AP}(\mathbb{R},\mathbb{R})  \times \mathcal{AP}_+(\mathbb{R},\mathbb{R}),
	\end{equation}
where $\frac{1}{p}$ is well-defined; see Lemma \ref{lm-1/f} \romannumeral2). Let $\phi(x)$ be any non-trivial solution of \x{sp-sl}. We consider a continuous branch of the following argument as
	\begin{equation*}
		\theta_\lambda\left(x;{\bf{v}} \right):=\arg\big(p(x)\phi^\prime(x)+\mathrm{i}\, \phi(x)\big).
	\end{equation*} 
Then we have the first main result.
\begin{thm} \label{thm-rn}
The limit
	\begin{equation*} 
		\lim_{x \to +\infty}\frac{\theta_\lambda(x;{\bf{v}})-\theta_\lambda(0;{\bf{v}})}{x}
	\end{equation*}
exists and is independent of the choice of solutions. We call it the rotation number of \xx{sp-sl}, and denote it by $\rho\left(\lambda,{\bf{v}}\right)$.
\end{thm}

\begin{rem}
{\rm  
The concept of the rotation number is due to Poincar\'e and is used to obtain a classification of orientation preserving self-homeomorphisms of the circle; see \cite{KH95}. Extensions of this concept have been considered by many authors in the literature; see, for example, \cite{BJ09, He83, Kw00, LL08}. 
}
\end{rem}

Similarly to $\mathcal{M}_f$, we denote the smallest additive group of $\R$ containing the Fourier exponents of $\textbf{v}$ by
	\begin{equation} \label{m-v}
		\mathcal{M}_{\textbf{v}}:= \left\{ \sum_{k=1}^m n_k \lambda_k \;\middle|\; \lambda_k \in \mathcal{M}_p \cup \mathcal{M}_q \cup \mathcal{M}_w, \, n_k \in \mathbb{Z}, \, m \in \mathbb{N} \right\}.
	\end{equation}
Note that $\mathcal{M}_{\frac{1}{p}}=\mathcal{M}_p$ for $p \in \mathcal{AP}_+(\mathbb{R},\mathbb{R})$; see Lemma \ref{lm-1/f} \romannumeral3). Then we have the gap labelling theorem as follows.
	
\begin{thm}\label{thm-gl}
Let\/ $\mathrm{J}$ be an open interval of\/ $\mathbb{R}\setminus \sigma(L_{\bf{v}})$. Then\/ $\rho(\lambda,\bf{v})$ is a constant in $\mathrm{J}$ and 
	\[
	2\rho(\lambda,{\bf{v}})\in \mathcal{M}_{\bf{v}} \qquad \mbox{for~}\lambda \in \mathrm{J}.
	\]
\end{thm}
	
\begin{rem}
{\rm  

\romannumeral1): Very recently, Damanik and his collaborators have conducted a series of studies on gap labelling for discrete one-dimensional ergodic Schr\"odinger/Jacobi operators using the Schwartzman homomorphism developed by Johnson; see, for example, \cite{DF2022, DF2023, DFZ2023, DEF2024}. In contrast, our work investigates gap labelling for almost periodic Sturm-Liouville operators directly in terms of Green's functions. In particular, we establish the almost periodicity of the Green's functions and exploit the module containment relationship between the coefficient functions and the Green's functions; see Lemma \ref{gg'ap}.

\romannumeral2): Let us revisit Figure \ref{figp} and Figure \ref{figap}. It is easy to verify that $\mathcal{M}_{\bf{v}}=\Z$ for Figure \ref{figp}. Thus it is reasonable that the platforms are only distributed at $\frac{\mathbb{Z}}{2}$. However, we have $\mathcal{M}_{\bf{v}}=\Z+\sqrt{2}\Z$ for Figure \ref{figap}. Due to the density of $\Z+\sqrt{2}\Z$, the phenomenon of the Devil's Staircase may appear in this case.
}
\end{rem}

Throughout this paper, we adopt the following notations:

\noindent $\bullet$ $\mathrm{i}$: the imaginary unit;

\noindent $\bullet$ $\mathrm{e}$: the Euler number;



\noindent $\bullet$ $\mathbb{K}$: either $\mathbb{R}$ or $\mathbb{C}$;

\noindent $\bullet$ $I$: the identity operator;

\noindent $\bullet$ $\cdot^*$: the conjugate of $\cdot$.  If $\cdot=f$ is a complex function, then $f^*$ denotes its conjugate function. If $\cdot=z$ is a complex number, then $z^*$ denotes its conjugate complex number;

\noindent $\bullet$ $f^{\prime}$: denote the derivative $\disp\frac{\mathrm{d}f}{\mathrm{d}x}$;

\noindent $\bullet$ $\mathcal{L}^p(D,\mathrm{d}\mu)$: the space of all $\mathcal{L}^p$-integrable functions $f$ defined on the set $D$ with the norm $\|f\|_{\mathcal{L}^p(D,\mathrm{d}\mu)}:=\left( \int_D|f|^p\mathrm{d}\mu\right) ^{\frac{1}{p}}$, for $p \in \mathbb{N}$;


\noindent $\bullet$ $\mathcal{C}(\mathbb{R}, \mathbb{K})$: the space of all continuous functions $f: \mathbb{R}\rightarrow\mathbb{K}$ with the norm $\|f\|_{\infty}:=\max_{x\in\mathbb{R}}|f(x)|$;

\noindent $\bullet$ $\mathcal{AC}(\mathbb{R},\mathrm{d}\mu)$: the space of all absolutely continuous functions with respect to the measure $\mu$, i.e., for any $f\in\mathcal{AC}(\mathbb{R},\mathrm{d}\mu)$, there exists a locally integrable function $g$ such that $f(x)=f(c)+\int_{c}^{x}g(t)\mathrm{d}\mu(t)$, $c\in\mathbb{R}$;


\noindent $\bullet$ $\sigma(L)$: the spectrum of the operator $L$;

\noindent $\bullet$ $\sharp\{\cdot\}$:  the number of elements in a given set.

\section{Almost periodic functions}\label{sec2}
\setcounter{equation}{0}
In this section, we state some fundamental results on almost periodic functions, and then introduce the triplets of almost periodic functions and the corresponding joint hulls which will be used later. For more details on almost periodic functions, see \cite{Fink}.

\subsection{Module}

Let $f\in \mathcal{AP}(\mathbb{R},\mathbb{K})$, and $\mathcal{M}_f$ be the frequency module of $f$.  $\mathcal{AP}_{\mathcal{M}_f}(\mathbb{R},\mathbb{K})$ denotes the space of all functions in $\mathcal{AP}(\mathbb{R},\mathbb{K})$ whose frequency module belongs to $\mathcal{M}_{f}$. $\mathcal{C}(\mathrm{E}(f),\mathbb{K})$ denotes the space of all continuous functions from $\mathrm{E}(f)$ to $\mathbb{K}$. 

\begin{lem} \label{lm-equi} {\rm \cite{Fink}} Let $f\in \mathcal{AP}(\mathbb{R},\mathbb{K})$.
	\begin{itemize}
		\item[{\rm \romannumeral1):}] If $f^{\prime}(x)$ is uniformly continuous on $\R$, then $f^{\prime}(x) \in \mathcal{AP}(\mathbb{R},\mathbb{K})$ and $\mathcal{M}_{f^{\prime}} \subset \mathcal{M}_f$.
		\item[{\rm \romannumeral2):}] $\mathcal{AP}_{\mathcal{M}_f}(\mathbb{R},\mathbb{K})$ is a Banach algebra.
		\item[{\rm \romannumeral3):}] $\mathcal{AP}_{\mathcal{M}_f}(\mathbb{R},\mathbb{K})$ and\/ $\mathcal{C}(\mathrm{E}(f),\mathbb{K})$ are isomorphic as Banach algebras.
	\end{itemize}	
\end{lem}

Moreover, we have

\begin{lem}\label{lm-1/f}
	Let $f \in \mathcal{AP}_+(\mathbb{R},\mathbb{R})$. 
	\begin{itemize}
		\item[{\rm \romannumeral1):}] There exists a constant $\delta:=\delta(f)>0$ such that $f(x)\geq \delta$  for all $x\in\mathbb{R}$.
		\item[{\rm \romannumeral2):}] $\frac{1}{f}$ is well-defined, and $\frac{1}{f}\in\mathcal{AP}_+(\mathbb{R},\mathbb{R})$.
		\item[{\rm \romannumeral3):}] $\mathcal{M}_{\frac{1}{f}}=\mathcal{M}_f$.
	\end{itemize}	
\end{lem}
\begin{proof}
\fbox{{\rm {\romannumeral1})}}\ : We extend $f(x)$ to $F(\xi) \in \mathcal{C}(\mathrm{E}(f),\R)$ such that $F(\xi_0 \cdot x)=f(x)$ where $\xi_0:=f$ and $x\in \R$. By \x{ap+} and Lemma \ref{lm-equi}, we have $F(\xi)>0$ for all $\xi\in \mathrm{E}(f)$. Thus there exists a constant $\delta>0$ such that $F(\xi)\geq \delta$ for all $\xi\in \mathrm{E}(f)$, because $\mathrm{E}(f)$ is compact. In particular, $f(x)\geq \delta>0$  for all $x\in\mathbb{R}$. 

\fbox{{\rm {\romannumeral2})}}\ : It is obvious by {\rm {\romannumeral1})} and the boundedness of Bohr almost periodic functions. 

\fbox{{\rm {\romannumeral3})}}\ : For any sequence $\{t_n\} \subset \R$ such that  $\lim\limits_{n\rightarrow\infty} f\cdot t_n=f_0$ exists uniformly on the real axis, we have
	\begin{equation*}
	f_0(x)\geq \delta>0,\qquad \mbox{for all~} x\in\mathbb{R},
	\end{equation*}
where \romannumeral1) is used. It follows that 
	\begin{equation*}
	\left|\frac{1}{f(x+t_n)}-\frac{1}{f_0(x)}\right|=\left|\frac{f(x+t_n)-f_0(x)}{f(x+t_n)f_0(x)}\right|\leq \frac{1}{\delta^2}\left|f(x+t_n)-f_0(x)\right|.
	\end{equation*} 
Hence $\lim\limits_{n\rightarrow\infty} \frac{1}{f}\cdot t_n=\frac{1}{f_0}$ exists uniformly on the real axis. By \cite[Theorem 4.5]{Fink}, we obtain $\mathcal{M}_{\frac{1}{f}}\subseteq \mathcal{M}_f$. Conversely, by a similar argument, we have $\mathcal{M}_f\subseteq \mathcal{M}_{\frac{1}{f}}$. The desired result is obtained.
\end{proof}

\begin{rem} \label{re-1/f}
	{\rm \romannumeral1): For any $\widetilde{f} \in \mathrm{E}(f)$, Lemma \ref{lm-1/f} holds as well, and $\delta$ is independent of the choice of $\widetilde{f}$.\\[1mm]
	\romannumeral2): In general, we do not have $\exp\big(\frac{1}{f}\big)=\exp(f)$. For example, let $f:=\cos(x)+2$. Then $\exp(f)=\{0,\pm1\}$, but $\exp\big(\frac{1}{f}\big)=\mathbb{Z}$.\\[1mm]
	\romannumeral3): If we consider the following space:
		\begin{equation*} 
		\mathcal{AP}_+(\mathbb{R},\mathbb{C}):=\Big\{f \in \mathcal{AP}(\mathbb{R},\mathbb{C}): \big|\widetilde{f}(x)\big| \; \text{is positive for each } \widetilde{f} \in  \mathrm{E}(f) \Big\},
		\end{equation*}
	then Lemma \ref{lm-1/f} still holds, and one only needs to replace $f(x)$ by $|f(x)|$ in \romannumeral1).}
\end{rem}

We state the following lemma which plays a fundamental role in the proof of the gap labelling theorem.
\begin{lem}\label{lm-fzero}
Assume that $f(x)\in \mathcal{AP}(\mathbb{R},\mathbb{R})$, $f^{\prime}(x)$ is uniformly continuous on $\R$, and any function $\widetilde{f}(x) \in \mathrm{E}(f)$ has only simple zeroes. Then the following limit exists and 
	\[
	\lim\limits_{x\rightarrow +\infty}\frac{\pi\sharp\{s\in [0,x):f(s)=0\}}{x}\in \mathcal{M}_f.
	\]
\end{lem}
\begin{proof}
By Lemma \ref{lm-equi} and Corollary \cite[p.412]{JM1982}, we have the desired result.
\end{proof}

\begin{rem} 
	{\rm The value of the limit can be regarded as the so-called \textsl{mean index}, which is a key quantity associated with non-periodic orbits in Hamiltonian systems arising from astromechanics; see \cite{HW20, Xia}.}
\end{rem}

Recall the notation \x{m-v}. Denote by $\mathcal{AP}_{\mathcal{M}_{\textbf{v}}}(\mathbb{R},\mathbb{K})$ the space of all functions in $\mathcal{AP}(\mathbb{R},\mathbb{K})$ whose frequency module belongs to $\mathcal{M}_{\textbf{v}}$. The relationship between the frequency modules of two elements in $\mathcal{AP}(\mathbb{R},\mathbb{K})$ has been established in \cite[Theorem 4.5]{Fink}. If we consider the vector-valued case instead of the scalar-valued case, the proof remains unchanged, except for the fact that the Fourier coefficients are vectors for the vector-valued case. Thus the relationship between the frequency modules of elements in $\mathcal{AP}(\mathbb{R},\mathbb{K}^3)$ and those in $\mathcal{AP}(\mathbb{R},\mathbb{K})$ can be stated as follows.

\begin{lem} \label{mv-eq}
	Let $f: \mathbb{R}\rightarrow\mathbb{K}$ be a continuous function. Then the following statements are equivalent:
	\begin{itemize}
		\item[{\rm \romannumeral1):}] $f \in \mathcal{AP}_{\mathcal{M}_{\bf{v}}}(\mathbb{R},\mathbb{K})$;
		\item[{\rm \romannumeral2):}] For any sequence $\{t_n\}_{n \in \N} \subset \R$ such that  $\lim\limits_{n\rightarrow\infty} \frac{1}{p}\cdot t_n$,  $\lim\limits_{n\rightarrow\infty} q \cdot t_n$ and $\lim\limits_{n\rightarrow\infty} w\cdot t_n$ exist uniformly on the real axis, it follows that $\lim\limits_{n\rightarrow\infty} f \cdot t_n$ also exists uniformly.
	\end{itemize}
\end{lem}

\subsection{Joint hull}

Recall the notation $\textbf{v} \in \mathcal{AP}_*(\R,\R^3)$ in \x{v}. $\mathcal{AP}_*(\R,\R^3)$ can be equipped with the metric 
	\begin{equation} \label{dist-v}
	\|\textbf{v}\|_{\infty}:=\max\left\lbrace \Big\|\frac{1}{p}\Big\|_{\infty}, \|q\|_{\infty}, \|w\|_{\infty}\right\rbrace.
	\end{equation}
Note that $\mathcal{AP}_*(\R,\R^3)$ is not complete, although $\mathcal{AP}(\R,\R^3)$ is a Banach space. We consider an $\R$ action on $\mathcal{AP}_*(\R,\R^3)$ by shifts, and denote for $\textbf{v} \in \mathcal{AP}_*(\R,\R^3)$ and $t \in \R$ the corresponding shifted element in $\mathcal{AP}_*(\R,\R^3)$ by 		
	\begin{equation}\label{sft-v}
	{\bf{v}}\cdot t:=\Big(\frac{1}{p\cdot t},q\cdot t, w\cdot t\Big).
	\end{equation} 
It is easy to verify that the $\R$ action satisfies the following conditions.
\begin{itemize}
	\item group structure:
		\begin{equation*}
		{\bf{v}} \cdot 0 = {\bf{v}}, \mbox{~and~} {\bf{v}} \cdot (t_1+t_2)= ({\bf{v}} \cdot t_1) \cdot t_2, \qquad \mbox{for all~} {\bf{v}} \in \mathcal{AP}_*(\R,\R^3), \ t_1, \ t_2 \in \R,
		\end{equation*}
	\item isometry:
		\begin{equation} \label{sft-is}
		\|{\bf{v}}_1 \cdot t-{\bf{v}}_2 \cdot t\|_{\infty}=\|{\bf{v}}_1-{\bf{v}}_2\|_{\infty}, \qquad \mbox{for all~} t \in \R, \; {\bf{v}}_i \in \mathcal{AP}_*(\R,\R^3),\ i=1,2,
		\end{equation}
	\item uniform continuity: for any $\epsilon > 0$, there exists a $\delta = \delta(\epsilon) > 0$ such that
		\begin{equation*} 
		\| {\bf{v}} \cdot t - {\bf{v}} \|_{\infty} < \epsilon, \qquad \mbox{for all~} {\bf{v}} \in \mathcal{AP}_*(\R,\R^3), \; t\in \R \text{ with } |t| < \delta.
		\end{equation*}
\end{itemize}
By \x{sft-is}, we have
\begin{lem} \label{orbs}
	The shifts $\{{\bf{v}} \cdot t \}_{t \in \R}$ are equicontinuous homeomorphisms.
\end{lem}
For $\textbf{v}=\Big(\frac{1}{p},q,w\Big) \in \mathcal{AP}_*(\R,\R^3)$, denote the \textsl{orbit} of $\textbf{v}$ by 
	\begin{equation*} 
	\mathrm{Orb}({\bf{v}}) := \{\textbf{v} \cdot t : t \in \R\}.
	\end{equation*}
We introduce
\begin{defn}\label{hull}
The joint hull of\/ ${\bf{v}} \in \mathcal{AP}_*(\R,\R^3)$ is defined by
	\begin{equation*}
	\mathrm{E}({\bf{v}}):=\overline{ \mathrm{Orb}({\bf{v}})}^{\|\cdot\|_{\infty}}.
	\end{equation*}
\end{defn}
Obviously, we have
	\begin{equation*}
	\mathrm{E}({\bf{v}}) \subset \mathrm{E}\left(\frac{1}{p}\right) \times \mathrm{E}(q) \times \mathrm{E}(w) \subset \mathcal{AP}_*(\R,\R^3) \subset \mathcal{AP}(\R,\R^3).
	\end{equation*}
By the definition of Bohr almost periodic functions and Lemma \ref{lm-1/f} \rmnum{2}), we obtain

\begin{lem}\label{E-cpt}
	$\mathrm{E}(\bf{v})$ is compact in $\mathcal{AP}(\R,\R^3)$ with the metric $\|\cdot\|_{\infty}$.
\end{lem}

One may equip $\mathrm{E}(\textbf{v})$ with a group structure as follows. Let
	\[
	\widetilde{\textbf{v}}_1=\lim\limits_{n\rightarrow\infty}\textbf{v}\cdot t_n,\quad \widetilde{\textbf{v}}_2=\lim\limits_{n\rightarrow\infty}\textbf{v}\cdot s_n
	\]
be elements of $\mathrm{E}(\textbf{v})$, then the product is defined by
	\begin{equation} \label{prdt}
	\widetilde{\textbf{v}}_1\cdot \widetilde{\textbf{v}}_2 := \lim\limits_{n\rightarrow\infty}\textbf{v}\cdot (t_n+s_n).
	\end{equation}
The limit is well-defined and the product is obviously commutative. The inverse is defined by
	\begin{equation} \label{invs}
	\widetilde{\textbf{v}}_1^{-1}:=\lim\limits_{n\rightarrow\infty}\textbf{v}\cdot (-t_n),
	\end{equation}
and $\textbf{v}$ is thus the identity element of the group $\mathrm{E}(\textbf{v})$.

\begin{lem} \label{atg} We have
	\begin{itemize}
		\item[{\rm \romannumeral1):}] $\mathrm{E}(\widetilde{\bf{v}})= \mathrm{E}(\bf{v})$, for each $\widetilde{\bf{v}}\in \mathrm{E}(\bf{v})$;
		\item[{\rm \romannumeral2):}] $\mathcal{M}_{\widetilde{\bf{v}}}=\mathcal{M}_{\bf{v}}$, for each $\widetilde{\bf{v}}\in \mathrm{E}(\bf{v})$;
		\item[{\rm \romannumeral3):}] $(\mathrm{E}({\bf{v}}),\cdot \, ,^{-1})$ is a compact abelian topological group;
		\item[{\rm \romannumeral4):}] the flow $\{\widetilde{\bf{v}} \cdot t \}_{t \in \R}$ on $\mathrm{E}(\bf{v})$ is uniquely ergodic with the Haar measure, denoted by $\mu_{\mathrm{E}(\bf{v})}$, being the only invariant measure; and
		\item[{\rm \romannumeral5):}] for any continuous function $f: \mathrm{E}(\bf{v}) \to \K$,
			\begin{equation*} \label{mv-atg}
			\lim_{x \to +\infty} \frac{1}{x} \int_{0}^{x} f(\widetilde{\bf{v}} \cdot t) \mathrm{d}t=\int_{\mathrm{E}(\bf{v})} f \mathrm{d} \mu_{\mathrm{E}(\bf{v})},
			\end{equation*}
		uniformly for all $\widetilde{\bf{v}} \in \mathrm{E}(\bf{v})$.
	\end{itemize}
\end{lem}

\begin{proof}
\fbox{{\rm {\romannumeral1})}}\ : Let $\widetilde{\bf{v}}\in\mathrm{E}(\bf{v})$. Since $\mathrm{E}(\bf{v})$ is invariant under the shift (\ref{sft-v}), we have $\mathrm{Orb}({\widetilde{\bf{v}}})\subseteq \mathrm{E}(\bf{v})$. By Lemma \ref{E-cpt}, we know that $\mathrm{E}(\bf{v})$ is closed. Thus $\mathrm{E}(\widetilde{\bf{v}})=\overline{ \mathrm{Orb}({\widetilde{\bf{v}}}) }^{\|\cdot\|_{\infty}}\subseteq \mathrm{E}(\bf{v})$. Conversely, assume that $\lim\limits_{n\rightarrow\infty}{\bf{v}}\cdot t_n=\widetilde{\bf{v}}\in\mathrm{E}(\bf{v})$. Then one has ${\bf{v}}=\lim\limits_{n\rightarrow\infty}\widetilde{\bf{v}}\cdot (-t_n)$ which implies $\bf{v}\in\mathrm{E}(\widetilde{\bf{v}})$. Thus by a similar argument, we have $\mathrm{E}(\widetilde{\bf{v}})= \mathrm{E}(\bf{v})$.

\fbox{{\rm {\romannumeral2})}}\ : Because of $\x{m-v}$, it is sufficient to show that for any fixed $f\in\mathcal{AP}(\mathbb{R},\mathbb{R})$, one has $\mathcal{M}_{\widetilde{f}}=\mathcal{M}_{f}$ for each $\widetilde{f}\in\mathrm{E}(f)$. We first claim that $\mathcal{M}_{\widetilde{f}}\subseteq \mathcal{M}_{f}$. In fact, assume that $\lim\limits_{n\rightarrow\infty}{f}\cdot t_n=\widetilde{f}\in\mathrm{E}(f)$. This implies that for any $\epsilon>0$, there exists an $N_0:=N_0(\epsilon)>0$ such that 
	\begin{equation} \label{atg1}
	|\widetilde{f}(x)-f(x+t_n)|<\epsilon, \qquad \mbox{for~all~} n>N_0 \mbox{~and~} x\in\mathbb{R}.
	\end{equation}
For any sequence $\{s_n\}_{n \in \N} \subset \R$ such that $\lim\limits_{n\rightarrow\infty} f\cdot s_n$ exists uniformly on the real axis, it follows from \x{atg1} that  $\lim\limits_{n\rightarrow\infty} \widetilde{f}\cdot s_n$ exists as well. By \cite[Theorem 4.5]{Fink}, we complete the proof of the claim. Conversely, if $\lim\limits_{n\rightarrow\infty}{f}\cdot t_n=\widetilde{f}\in\mathrm{E}(f)$, then we have $f=\lim\limits_{n\rightarrow+\infty}\widetilde{f}\cdot (-t_n)$. Along line with a similar argument, we have $\mathcal{M}_{f}\subseteq\mathcal{M}_{\widetilde{f}}$, and have the desired result.

\fbox{{\rm {\romannumeral3})}}\ : By \x{sft-is}, we may obtain that operations of the product \x{prdt} and the inverse \x{invs} are continuous on $\mathrm{E}({\bf{v}})$. 

\fbox{{\rm {\romannumeral4})}}\ : By \x{prdt}, we have $\widetilde{\bf{v}} \cdot t= ({\bf{v}}\cdot t) \cdot \widetilde{\bf{v}}$ for any $\widetilde{\bf{v}}\in\mathrm{E}(\bf{v})$ and $t \in \mathbb{R}$. This implies that the shift $\widetilde{\bf{v}} \cdot t$ can be regarded as a rotation on $\mathrm{E}(\bf{v})$. Then the result is deduced from {\romannumeral1}), {\romannumeral3}) and \cite[Theorem 6.20]{Wa82}.

\fbox{{\rm {\romannumeral5})}}\ : The result is deduced from {\romannumeral4}) and the unique ergodic theorem \cite[Theorem 6.19]{Wa82}.
\end{proof}

\section{Almost periodic Sturm-Liouville operators} \label{sec-sl}
\setcounter{equation}{0}

As mentioned in the Introduction, it is not sufficient to study the almost periodic Sturm-Liouville operator on its own. The operator \x{sl} needs be embedded into a family of operators as follows.

\subsection{Operators associated with an element of the hull $\mathrm{E}(\textbf{v})$}\label{ss-slh}

Let $\textbf{v} \in \mathcal{AP}_*(\R,\R^3)$. For any $\widetilde{\textbf{v}}:=\Big(\frac{1}{\widetilde p},\widetilde q,\widetilde w\Big)\in \mathrm{E}(\textbf{v})$, consider the Hilbert space $$\mathcal{L}^2(\mathbb{R},\widetilde{w}(x)\mathrm{d}x):=\left\lbrace f:\disp\int_{\mathbb{R}}|f(x)|^2\widetilde{w}(x)\mathrm{d}x<\infty \right\rbrace $$ with the following inner product
$$
\left\langle f,g\right\rangle_{\widetilde{\textbf{v}}} :=\int_{\mathbb{R}}f(x)^*g(x)\widetilde{w}(x)\mathrm{d}x\quad \;\text{for all}\; f,g\in \mathcal{L}^2(\mathbb{R},\widetilde{w}(x)\mathrm{d}x),
$$
and the Sturm-Liouville differential expression 
	\begin{equation} \label{sl-ex}
	\tau_{\widetilde{\textbf{v}}}f:=\frac{1}{\widetilde{w}(x)}\left( -\frac{\mathrm{d}}{\mathrm{d}x}\widetilde{p}(x)\frac{\mathrm{d}f(x)}{\mathrm{d}x}+\widetilde{q}(x)f(x)\right)\quad \;\text{for all}\; f \in \mathcal{D}(\tau_{\widetilde{\textbf{v}}}),
	\end{equation}
where $\mathcal{D}(\tau_{\widetilde{\textbf{v}}}):=\left\lbrace f\in \mathcal{L}^2(\mathbb{R},\widetilde{w}(x)\mathrm{d}x):f,\widetilde{p}f^{\prime}\in \mathcal{AC}(\mathbb{R},\mathrm{d}x), \tau_{\widetilde{\textbf{v}}}f\in \mathcal{L}^2(\mathbb{R},\widetilde{w}(x)\mathrm{d}x) \right\rbrace$ is the \textsl{maximal domain} of $\tau_{\widetilde{\textbf{v}}}$. Note that all items depend on $\widetilde{\textbf{v}}$. For $f,g\in\mathcal{D}(\tau_{\widetilde{\textbf{v}}})$ and $-\infty<c<d<+\infty$, using integration by parts twice, we obtain the Lagrange identity as follows.
	\begin{equation*}
	\int_{c}^{d}g(x)^*\tau_{\widetilde{\textbf{v}}} f(x)\widetilde{w}(x)\mathrm{d}x=\mathrm{W}_c(g^*,f;\widetilde{\textbf{v}})-\mathrm{W}_d(g^*,f;\widetilde{\textbf{v}})+\int_{c}^{d}\tau_{\widetilde{\textbf{v}}} g(x)^*f(x)\widetilde{w}(x)\mathrm{d}x,
	\end{equation*}
where 
	\begin{equation}\label{mwron}
	\mathrm{W}_x(f_1,f_2;\widetilde{\textbf{v}}):=\Big(f_1(\widetilde{p}f_2^\prime)-(\widetilde{p}f_1^\prime) f_2\Big)(x)
	\end{equation}
is called the \textsl{modified Wronskian} associated with $\widetilde{\textbf{v}}$. Taking the limit $c\rightarrow-\infty$ and $d\rightarrow+\infty$, one has
	\begin{equation*}
	\left\langle g,\tau_{\widetilde{\textbf{v}}} f\right\rangle_{\widetilde{\textbf{v}}}=\mathrm{W}_{-\infty}(g^*,f;\widetilde{\textbf{v}})-\mathrm{W}_{+\infty}(g^*,f;\widetilde{\textbf{v}})+\left\langle \tau_{\widetilde{\textbf{v}}} g,f\right\rangle_{\widetilde{\textbf{v}}},
	\end{equation*}
where $\mathrm{W}_{\pm\infty}(g^*,f;\widetilde{\textbf{v}})$ is regarded as a limit. For any $z \in\mathbb{C}$ and any two solutions $\phi_1$ and $\phi_2$ of $\tau_{\widetilde{\textbf{v}}} \phi=z\phi$, we know that
$$
\mathrm{W}(\phi_1,\phi_2;\widetilde{\textbf{v}}):=\mathrm{W}_x(\phi_1,\phi_2;\widetilde{\textbf{v}})
$$ 
is independent of $x$. Moreover, $\mathrm{W}(\phi_1,\phi_2;\widetilde{\textbf{v}})\neq 0$ if and only if $\phi_1$ and $\phi_2$ are linearly independent.

The differential operator $\tau_{\widetilde{\textbf{v}}}$ is called \textsl{limit point} at $+\infty$ if there exists a $z_0\in\mathbb{C}$ such that at least one solution of $\tau_{\widetilde{\textbf{v}}} \phi=z_0\phi$ does not belong to $\mathcal{L}^2(\mathbb{R}_+,\widetilde{w}(x)\mathrm{d}x)$. A similar definition applies at $-\infty$. If $\tau_{\widetilde{\textbf{v}}}$ is limit point at both $+\infty$ and $-\infty$, then $\tau_{\widetilde{\textbf{v}}}$ is called to be \textsl{limit point}. The concept of the limit point was introduced by H. Weyl. For more details, see \cite[Theorem 10.1.1]{Hille}. Due to \x{v} and \cite[Theorem 6.3]{J1987}, we have

\begin{lem}\label{lp}
For any $\widetilde{\bf{v}} \in \mathrm{E}(\bf{v})$, the differential expression	$\tau_{\widetilde{\bf{v}}}$ defined by \xx{sl-ex} is limit point.
\end{lem}

Furthermore, it follows from \cite[Lemma 9.4 and Theorem 9.6]{T2014} that

\begin{lem}\label{sa}
The Sturm-Liouville operator, defined by 
	\begin{equation*} 
	\begin{split}
	L_{\widetilde{\bf{v}}}:\quad \mathcal{D}(L_{\widetilde{\bf{v}}})&\rightarrow \mathcal{L}^2(\mathbb{R},\widetilde{w}(x)\mathrm{d}x)\\
	f&\mapsto \tau_{\widetilde{\bf{v}}}f,
	\end{split}
	\end{equation*}
is densely defined and self-adjoint on $\mathcal{L}^2(\mathbb{R},\widetilde{w}(x)\mathrm{d}x)$, where $\mathcal{D}(L_{\widetilde{\bf{v}}}):=\mathcal{D}(\tau_{\widetilde{\bf{v}}})$.
\end{lem}

Since a family of operators is under consideration, we recall a notion of convergence of operators. For each $n \in \mathbb{N}_0:=\mathbb{N}\cup\{0\}$, let $L_n$ be a self-adjoint operator on a Hilbert space $\mathcal{H}_n$. Define
\begin{equation*}
\begin{split}
J_n:\quad \mathcal{H}_0&\rightarrow \mathcal{H}_n\\
f&\mapsto f
\end{split}
\end{equation*}
We say  $L_n\rightarrow L_0$ \textsl{in generalized norm resolvent sense} if there exists a $z_0\in\mathbb{C}$ satisfying $\Im z_0\neq 0$, such that
	\begin{equation} \label{nrs}
	\big\|J_n^*(L_n-z_0I)^{-1}J_n-(L_0-z_0I)^{-1}\big\|\rightarrow 0,\qquad\text{as $n\rightarrow\infty$}.
	\end{equation}
For more details, see \cite{TWXZ}.

\begin{lem}\label{spe-inv}
	For any ${\widetilde{\bf{v}}}\in\mathrm{E}({\bf{v}})$, $\sigma(L_{{\widetilde{\bf{v}}}})=\sigma(L_{{\bf{v}}})$.
\end{lem}
\begin{proof}
For any ${\widetilde{\bf{v}}}\in\mathrm{E}({\bf{v}})$, there exists $\{t_n\}\subseteq \mathbb{R}$ such that $\textbf{v}\cdot t_n\rightarrow {\widetilde{\bf{v}}}$. Thus by \cite[Theorem 2.10]{TWXZ}, we have $\sigma(L_{{\widetilde{\bf{v}}}})=\lim\limits_{n\rightarrow\infty}\sigma(L_{\textbf{v}\cdot t_n})$, where $\lim\limits_{n\rightarrow\infty}\sigma(L_{\textbf{v}\cdot t_n})$ denotes the set of all $\lambda$ for which there is a sequence $\lambda_n\in\sigma(L_{\textbf{v}\cdot t_n})$ converging to $\lambda$. Since $L_{\textbf{v}\cdot t_n}$ is a translation of $L_{{\bf{v}}}$, we have $\sigma(L_{\textbf{v}\cdot t_n})=\sigma(L_{{\bf{v}}})$. Therefore, $\sigma(L_{{\widetilde{\bf{v}}}})=\sigma(L_{{\bf{v}}})$.
\end{proof}

\subsection{Green's function}

The so-called \textsl{Green's function} is a fundamental tool for studying the spectral theory of Sturm-Liouville operators. For more details, see \cite{S1965,T2014,GN2024}. For $\widetilde{\textbf{v}}=\big(\frac{1}{\widetilde{p}},\widetilde{q},\widetilde{w}\big)$ and $z\in \mathbb{C}\setminus\sigma(L_{\widetilde{\bf{v}}})$, we consider the differential equation
	\begin{equation}\label{tauz}
	\tau_{\widetilde{\textbf{v}}}\phi(x)=\frac{1}{\widetilde{w}(x)}\left( -\frac{\mathrm{d}}{\mathrm{d}x}\widetilde{p}(x)\frac{\mathrm{d}\phi(x)}{\mathrm{d}x}+\widetilde{q}(x)\phi(x)\right)=z\phi(x).
	\end{equation}
By Lemma \ref{lp}, up to a constant multiple, equation (\ref{tauz}) has a unique nontrivial solution $\phi_+(x,z;\widetilde{\textbf{v}})\in \mathcal{L}^2(\R_+,\widetilde{w}(x)\mathrm{d}x)$ and a unique nontrivial solution $\phi_-(x,z;\widetilde{\textbf{v}})\in \mathcal{L}^2(\R_-,\widetilde{w}(x)\mathrm{d}x)$, which are known as \textsl{Weyl's solutions}; see \cite[Theorem 9.9]{T2014}. 
Then the \textsl{Green's function} for $L_{\widetilde{\textbf{v}}}$ is defined by
	\begin{equation} \label{green}
	G(x,y,z;\widetilde{\textbf{v}}):=\frac{1}{\mathrm{W}(\phi_+(\cdot,z;\widetilde{\textbf{v}}),\phi_-(\cdot,z;\widetilde{\textbf{v}});\widetilde{\textbf{v}})}\left\{
	\begin{array}{c}
	\phi_+(x,z;\widetilde{\textbf{v}})\phi_-(y,z;\widetilde{\textbf{v}}),\quad x\geq y,\\
	\phi_-(x,z;\widetilde{\textbf{v}})\phi_+(y,z;\widetilde{\textbf{v}}),\quad x\leq y. \end{array}
	\right.
	\end{equation}
Here $G(x,y,z;\widetilde{\textbf{v}})$ depends on $\widetilde{\textbf{v}}$ as well.
\begin{lem}\label{sft-g}
Let $\widetilde{\bf{v}} \in \mathrm{E}(\bf{v})$ and $z\in \mathbb{C}\setminus\sigma(L_{\widetilde{\bf{v}}})$ be fixed. We have
	\begin{itemize}
		\item[{\rm \romannumeral1):}] $G(x,y,z;\widetilde{\bf{v}})$ is continuous with respect to $x \in \mathbb{R}$ and $y\in \mathbb{R}$.
		\item[{\rm \romannumeral2):}] $G(x+t,y+t,z;\widetilde{\bf{v}})=G(x,y,z;\widetilde{\bf{v}}\cdot t)$, for all $t\in\mathbb{R}$.
	\end{itemize}
\end{lem} 
\begin{proof}

\fbox{{\rm {\romannumeral1})}}\ : It is obvious by \x{green} and the continuity of $\phi_{\pm}(\cdot,z;\widetilde{\textbf{v}})$.

\fbox{{\rm {\romannumeral2})}}\ : 
For any $t \in \R$, since both $\phi_{\pm}(x+t,z;\widetilde{\textbf{v}})$ and $\phi_{\pm}(x,z;\widetilde{\textbf{v}}\cdot t)$ are solutions of $\tau_{\widetilde{\textbf{v}}\cdot t} \phi= z \phi$ in $\mathcal{L}^2(\R_{\pm},\widetilde{w}\cdot t(x)\mathrm{d}x)$, by the uniqueness (up to a constant multiple) of solutions in $\mathcal{L}^2(\R_{\pm},\widetilde{w}\cdot t(x)\mathrm{d}x)$, we have 
	\begin{equation} \label{uni-phi}
	\phi_{\pm}(x+t,z;\widetilde{\textbf{v}})=k_{\pm}(z;\widetilde{\textbf{v}})\phi_{\pm}(x,z;\widetilde{\textbf{v}}\cdot t),
	\end{equation}
where $k_{\pm}(z;\widetilde{\textbf{v}})$ are non-zero constants. 
Hence for $x\geq y$ and $t \in \mathbb{R}$, we have
	\begin{equation*}
	\begin{split}
	&~~~~G(x+t,y+t,z;\widetilde{\textbf{v}})\\
	&=\frac{\phi_+(x+t,z;\widetilde{\textbf{v}})\phi_-(y+t,z;\widetilde{\textbf{v}})}{\widetilde{p}(x+t)\big(\phi_+(x+t,z;\widetilde{\textbf{v}})\phi_-^\prime(x+t,z;\widetilde{\textbf{v}})-\phi_+^\prime(x+t,z;\widetilde{\textbf{v}})\phi_-(x+t,z;\widetilde{\textbf{v}})\big)} \quad (\text{by \x{green}, \x{mwron}})\\
	&=\frac{\phi_{+}(x,z;\widetilde{\textbf{v}}\cdot t)\phi_{-}(y,z;\widetilde{\textbf{v}}\cdot t)}{\widetilde{p}\cdot t(x)\big( \phi_{+}(x,z;\widetilde{\textbf{v}}\cdot t)\phi_{-}^\prime(x,z;\widetilde{\textbf{v}}\cdot t) -\phi_{+}^\prime(x,z;\widetilde{\textbf{v}}\cdot t)\phi_{-}(x,z;\widetilde{\textbf{v}}\cdot t) \big)}\quad (\text{by \x{uni-phi}})\\
	&=G(x,y,z;\widetilde{\textbf{v}}\cdot t) \quad(\text{by \x{mwron}, \x{green}}).
	\end{split}
	\end{equation*}
The similar argument applies to the case where $x\leq y$. The proof is complete. \end{proof}

\begin{rem}
{\rm Let $z \in \mathbb{C}\setminus \mathbb{R}$ be fixed and take $x=y$ in Lemma \ref{sft-g}. Then the map 
	\[
	(x,\widetilde{\textbf{v}}) \mapsto G(x,x,z;\widetilde{\textbf{v}})
	\] 
defines a cocycle over the shift \x{sft-v} on $\mathbb{R} \times \mathrm{E}(\bf{v})$. An interesting application is the following. We apply Lemma \ref{atg} {\rm {\romannumeral4})} to the observation $f_z(\widetilde{\textbf{v}})=G(0,0,z;\widetilde{\textbf{v}}) \in \mathbb{C}$. Then we have 
	\begin{equation*}
		\begin{split} \mathrm{M}_x\big(G(x,x,z;\widetilde{\textbf{v}})\big)=\int_{\mathrm{E}(\textbf{v})}G(0,0;z;\widetilde{\textbf{v}})\mathrm{d}\mu_{\mathrm{E}(\bf{v})},
		\end{split}
	\end{equation*}
where $\mu_{\mathrm{E}(\bf{v})}$ is the Haar measure on $\mathrm{E}(\bf{v})$.}
\end{rem}

According to \cite[Section 9]{J1987}, the Green's function  can also be written as follows. For $ i=1,2$, suppose $u_i:=u_i(x,z;\widetilde{\textbf{v}})$ are two linearly independent solutions of \x{tauz} with
	\begin{equation*}
	\begin{pmatrix}
	u_1(0,z;\widetilde{\textbf{v}}) & u_2(0,z;\widetilde{\textbf{v}})\\
	\widetilde{p}(0)u_1^{\prime}(0,z;\widetilde{\textbf{v}}) & \widetilde{p}(0)u_2^{\prime}(0,z;\widetilde{\textbf{v}})
	\end{pmatrix}=\begin{pmatrix}
	1 & 0\\
	0 & 1
	\end{pmatrix}.
	\end{equation*} 
Due to the linear independence of $u_1$ and $u_2$, one has
	\begin{equation} \label{upm2}
	\left\{
	\begin{array}{c}
	\phi_+(x,z;\widetilde{\textbf{v}}):=m_{11}(z;\widetilde{\textbf{v}})u_1(x,z;\widetilde{\textbf{v}})+m_{12}(z;\widetilde{\textbf{v}})u_2(x,z;\widetilde{\textbf{v}}),\\
	\phi_-(x,z;\widetilde{\textbf{v}}):=m_{21}(z;\widetilde{\textbf{v}})u_1(x,z;\widetilde{\textbf{v}})+m_{22}(z;\widetilde{\textbf{v}})u_2(x,z;\widetilde{\textbf{v}}),\end{array}
	\right.
	\end{equation}
where $m_{jk}(z;\widetilde{\textbf{v}})$ are constants depending on $z$ and $\widetilde{\textbf{v}}$, for $j,k=1,2$. Beacause of the uniqueness (up to a constant multiple) of solutions in $\mathcal{L}^2(\R_{\pm},\widetilde{w}(x)\mathrm{d}x)$ , we may assume that
	\begin{equation*}
	\mathrm{W}(\phi_+(\cdot,z;\widetilde{\textbf{v}}),\phi_-(\cdot,z;\widetilde{\textbf{v}});\widetilde{\textbf{v}})=1.
	\end{equation*}
Then it follows from \x{green} and \x{upm2} that
	\begin{equation} \label{green2}
	G(x,y,z;\widetilde{\textbf{v}})=\left\{
	\begin{array}{c}
	\sum_{j,k=1}^{2}m_{jk}^+(z;\widetilde{\textbf{v}})u_j(x,z;\widetilde{\textbf{v}})u_k(y,z;\widetilde{\textbf{v}}),\quad x\geq y,\\
	\\
	\sum_{j,k=1}^{2}m_{jk}^-(z;\widetilde{\textbf{v}})u_j(x,z;\widetilde{\textbf{v}})u_k(y,z;\widetilde{\textbf{v}}),\quad x\leq y, \end{array}
	\right.
	\end{equation}
where
	\[
	m_{11}^{\pm}(z;\widetilde{\textbf{v}}):=m_{11}(z;\widetilde{\textbf{v}})m_{21}(z;\widetilde{\textbf{v}}),\quad m_{22}^{\pm}(z;\widetilde{\textbf{v}}):=m_{12}(z;\widetilde{\textbf{v}})m_{22}(z;\widetilde{\textbf{v}}),
	\] 
	\[
	m_{12}^{+}(z;\widetilde{\textbf{v}})=m_{21}^{-}(z;\widetilde{\textbf{v}}):=m_{11}(z;\widetilde{\textbf{v}})m_{22}(z;\widetilde{\textbf{v}}),\quad m_{21}^{+}(z;\widetilde{\textbf{v}})=m_{12}^{-}(z;\widetilde{\textbf{v}}):=m_{12}(z;\widetilde{\textbf{v}})m_{21}(z;\widetilde{\textbf{v}}).
	\]
	
Based on \x{green2}, we derive the following continuity results.

\begin{lem}\label{mgg'-con}
For $n \in \mathbb{N}_0$, let ${\widetilde{\bf{v}}}_n \in \mathrm{E}({\bf{v}})$ and $z\in\mathbb{C}\setminus\sigma(L_{{\bf{v}}})$. Suppose that ${\widetilde{\bf{v}}}_n\rightarrow{\widetilde{\bf{v}}}_0$ and $L_{{\widetilde{\bf{v}}}_n}\rightarrow L_{{\widetilde{\bf{v}}}_0}$ in generalized norm resolvent sense. Then, as $n \to +\infty$, we have
	\begin{itemize}
		\item[{\rm \romannumeral1):}] $m_{jk}^{\pm}(z;{\widetilde{\bf{v}}}_n)\rightarrow m_{jk}^{\pm}(z;{\widetilde{\bf{v}}}_0)$, \quad $j,\, k=1,2$.
		\item[{\rm \romannumeral2):}] $G(x,y,z;{\widetilde{\bf{v}}}_n)\rightarrow G(x,y,z;{\widetilde{\bf{v}}}_0)$  uniformly on any finite interval of $x \in \mathbb{R}$ and $y\in \mathbb{R}$.
		\item[{\rm \romannumeral3):}] $\frac{\mathrm{d}}{\mathrm{d}x}G(x,x,z;{\widetilde{\bf{v}}}_n)\rightarrow \frac{\mathrm{d}}{\mathrm{d}x}G(x,x,z;{\widetilde{\bf{v}}}_0)$ 
		uniformly on any finite interval of $x\in \mathbb{R}$.
	\end{itemize}
\end{lem}
\begin{proof}

\fbox{{\rm {\romannumeral1})}}\ : Let $\widetilde{{\bf{v}}}\in\mathrm{E}(\textbf{v})$. For any $z\in\mathbb{C}\setminus\sigma(L_{{\bf{v}}})$ and $-\infty<a_1<b_1<a_2<b_2<+\infty$, denote
	\begin{equation}\label{sol-d}
	u_{kj}(\cdot,z;\widetilde{\textbf{v}}):=u_{k}(\cdot,z;\widetilde{\textbf{v}})\big|_{(a_j, b_j)},\quad k=1,2.
	\end{equation}
Since $\widetilde{{\bf{v}}}$ is real-valued, we have 
	\begin{equation}\label{sol-ad}
	u_{k}(\cdot,z;\widetilde{\textbf{v}})=u_{k}(\cdot,z^*;\widetilde{\textbf{v}})^*,\quad k=1,2.
	\end{equation}
Then for $n,\, i=1,2$, one has
	\begin{equation*}
	\begin{split}
	&\left\langle u_{n2}(\cdot ,z;\widetilde{\textbf{v}}), (L_{\widetilde{\textbf{v}}}-zI)^{-1} u_{i1}(\cdot ,z^*;\widetilde{\textbf{v}})\right\rangle_{\widetilde{\textbf{v}}}\\
	=&\int_{-\infty}^{+\infty}u_{n2}(x ,z;\widetilde{\textbf{v}})^*(L_{\widetilde{\textbf{v}}}-zI)^{-1} u_{i1}(x ,z^*;\widetilde{\textbf{v}})\widetilde{w}(x)\mathrm{d}x\\
	=&\int_{-\infty}^{+\infty}u_{n2}(x ,z;\widetilde{\textbf{v}})^*\int_{-\infty}^{+\infty}G(x,y,z;\widetilde{\textbf{v}})u_{i1}(y ,z^*;\widetilde{\textbf{v}})\widetilde{w}(y)\mathrm{d}y\; \widetilde{w}(x)\mathrm{d}x\\
	=&\int_{a_2}^{b_2}u_{n2}(x ,z;\widetilde{\textbf{v}})^*\int_{a_1}^{b_1}G(x,y,z;\widetilde{\textbf{v}})u_{i1}(y ,z^*;\widetilde{\textbf{v}})\widetilde{w}(y)\mathrm{d}y\; \widetilde{w}(x)\mathrm{d}x\qquad (\text{by \x{sol-d}})\\
	=&\int_{a_2}^{b_2}u_{n2}(x ,z;\widetilde{\textbf{v}})^*\int_{a_1}^{b_1}\bigg(\sum_{j,k=1}^{2}m_{jk}^+(z;\widetilde{\textbf{v}})u_j(x,z;\widetilde{\textbf{v}})u_k(y,z;\widetilde{\textbf{v}})\bigg)u_{i1}(y ,z^*;\widetilde{\textbf{v}})\widetilde{w}(y)\mathrm{d}y\; \widetilde{w}(x)\mathrm{d}x\\
	&\qquad\qquad\qquad\qquad\qquad\qquad\qquad\qquad\qquad\qquad\qquad\qquad\qquad\qquad\qquad\qquad\qquad (\text{by \x{green2}})\\
	=&\sum_{j,k=1}^{2}m_{jk}^+(z;\widetilde{\textbf{v}})\int_{a_2}^{b_2}u_{n2}(x ,z;\widetilde{\textbf{v}})^*u_j(x,z;\widetilde{\textbf{v}})\widetilde{w}(x)\mathrm{d}x\int_{a_1}^{b_1}u_k(y,z;\widetilde{\textbf{v}})u_{i1}(y ,z^*;\widetilde{\textbf{v}})\widetilde{w}(y)\mathrm{d}y\\
	=&\sum_{j,k=1}^{2}m_{jk}^+(z;\widetilde{\textbf{v}})\int_{-\infty}^{+\infty}u_{n2}(x ,z;\widetilde{\textbf{v}})^*u_j(x,z;\widetilde{\textbf{v}})\widetilde{w}(x)\mathrm{d}x\int_{-\infty}^{+\infty}u_k(y,z^*;\widetilde{\textbf{v}})^*u_{i1}(y ,z^*;\widetilde{\textbf{v}})\widetilde{w}(y)\mathrm{d}y\\
	&\qquad\qquad\qquad\qquad\qquad\qquad\qquad\qquad\qquad\qquad\qquad\qquad\qquad\qquad \qquad (\text{by \x{sol-d}}\; {\rm and} \;\x{sol-ad})\\
	=&\sum_{j,k=1}^{2}m_{jk}^+(z;\widetilde{\textbf{v}})\left\langle u_{n2}(\cdot ,z;\widetilde{\textbf{v}}),  u_{j2}(\cdot ,z;\widetilde{\textbf{v}})\right\rangle_{\widetilde{\textbf{v}}}\left\langle u_{k1}(\cdot ,z^*;\widetilde{\textbf{v}}),  u_{i1}(\cdot ,z^*;\widetilde{\textbf{v}})\right\rangle_{\widetilde{\textbf{v}}}.
	\end{split}
	\end{equation*}
That is to say,	$$U(z;\widetilde{\textbf{v}})=U_2(z;\widetilde{\textbf{v}})M^+(z;\widetilde{\textbf{v}})U_1(z;\widetilde{\textbf{v}}),$$
where 
	\begin{equation*}
	\begin{split}
	& U(z;\widetilde{\textbf{v}}):=\begin{pmatrix}
	\left\langle u_{12}(\cdot ,z;\widetilde{\textbf{v}}), (L_{\widetilde{\textbf{v}}}-zI)^{-1} u_{11}(\cdot ,z^*;\widetilde{\textbf{v}})\right\rangle_{\widetilde{\textbf{v}}} & \left\langle u_{12}(\cdot ,z;\widetilde{\textbf{v}}), (L_{\widetilde{\textbf{v}}}-zI)^{-1} u_{21}(\cdot ,z^*;\widetilde{\textbf{v}})\right\rangle_{\widetilde{\textbf{v}}}\\
	\left\langle u_{22}(\cdot ,z;\widetilde{\textbf{v}}), (L_{\widetilde{\textbf{v}}}-zI)^{-1} u_{11}(\cdot ,z^*;\widetilde{\textbf{v}})\right\rangle_{\widetilde{\textbf{v}}} & \left\langle u_{22}(\cdot ,z;\widetilde{\textbf{v}}), (L_{\widetilde{\textbf{v}}}-zI)^{-1} u_{21}(\cdot ,z^*;\widetilde{\textbf{v}})\right\rangle_{\widetilde{\textbf{v}}}
	\end{pmatrix},\\
	&U_1(z;\widetilde{\textbf{v}}):=\begin{pmatrix}
	\left\langle u_{11}(\cdot ,z^*;\widetilde{\textbf{v}}),  u_{11}(\cdot ,z^*;\widetilde{\textbf{v}})\right\rangle_{\widetilde{\textbf{v}}} & \left\langle u_{11}(\cdot ,z^*;\widetilde{\textbf{v}}),  u_{21}(\cdot ,z^*;\widetilde{\textbf{v}})\right\rangle_{\widetilde{\textbf{v}}}\\
	\left\langle u_{21}(\cdot ,z^*;\widetilde{\textbf{v}}),  u_{11}(\cdot ,z^*;\widetilde{\textbf{v}})\right\rangle_{\widetilde{\textbf{v}}} & \left\langle u_{21}(\cdot ,z^*;\widetilde{\textbf{v}}),  u_{21}(\cdot ,z^*;\widetilde{\textbf{v}})\right\rangle_{\widetilde{\textbf{v}}}
	\end{pmatrix},\\
	&M^+(z;\widetilde{\textbf{v}}):=\begin{pmatrix}
	m_{11}^+(z;\widetilde{\textbf{v}}) & m_{12}^+(z;\widetilde{\textbf{v}})\\
	m_{21}^+(z;\widetilde{\textbf{v}}) & m_{22}^+(z;\widetilde{\textbf{v}})
	\end{pmatrix},\\
	&U_2(z;\widetilde{\textbf{v}}):=\begin{pmatrix}
	\left\langle u_{12}(\cdot ,z;\widetilde{\textbf{v}}),  u_{12}(\cdot ,z;\widetilde{\textbf{v}})\right\rangle_{\widetilde{\textbf{v}}} & \left\langle u_{12}(\cdot ,z;\widetilde{\textbf{v}}),  u_{22}(\cdot ,z;\widetilde{\textbf{v}})\right\rangle_{\widetilde{\textbf{v}}}\\
	\left\langle u_{22}(\cdot ,z;\widetilde{\textbf{v}}),  u_{12}(\cdot ,z;\widetilde{\textbf{v}})\right\rangle_{\widetilde{\textbf{v}}} & \left\langle u_{22}(\cdot ,z;\widetilde{\textbf{v}}),  u_{22}(\cdot ,z;\widetilde{\textbf{v}})\right\rangle_{\widetilde{\textbf{v}}}
	\end{pmatrix}.
	\end{split}
	\end{equation*}
Since $U_1(z;\widetilde{\textbf{v}})$ is the Gram matrix of the linearly independent solutions $u_{1}(\cdot ,z^*;\widetilde{\textbf{v}})$ and $u_{2}(\cdot ,z^*;\widetilde{\textbf{v}})$, it is non-singular. The same holds for $U_2(z;\widetilde{\textbf{v}})$. We then have
	\begin{equation}\label{M+}
	M^+(z;\widetilde{\textbf{v}})=U_2(z;\widetilde{\textbf{v}})^{-1}U(z;\widetilde{\textbf{v}})U_1(z;\widetilde{\textbf{v}})^{-1}.
	\end{equation}
Since
	\[
	\|\widetilde{\bf{v}}_n\|_{\infty}=\|\textbf{v}\|_{\infty}<\infty,\qquad n\in\mathbb{N}_0,
	\]
and 
	\[
	\|\widetilde{\bf{v}}_n-\widetilde{\bf{v}}_0\|_{\infty}\rightarrow 0,\qquad \text{as $n\rightarrow \infty$,}
	\]
by the initial value theory of ODEs, we have
	\begin{equation}\label{sol-conv}
	\begin{pmatrix}
	u_i(x,z;\widetilde{\bf{v}}_n)\\
	p_n(x)u_i^{\prime}(x,z;\widetilde{\bf{v}}_n)
	\end{pmatrix}\rightarrow\begin{pmatrix}
	u_i(x,z;\widetilde{\bf{v}}_0)\\
	p_0(x)u_i^{\prime}(x,z;\widetilde{\bf{v}}_0)
	\end{pmatrix},\qquad \text{as $n\rightarrow \infty$, for~} i=1,2,
	\end{equation}
uniformly on any finite interval of $x \in \mathbb{R}$. It follows that 
	\[
	U_i(z;\widetilde{\bf{v}}_n)\rightarrow U_i(z;\widetilde{\bf{v}}_0), \qquad \text{as $n\rightarrow \infty$, for~} i=1,2.
	\]
Combining with the generalized norm resolvent convergence \x{nrs}, we have 
	\[
	U(z;\widetilde{\bf{v}}_n)\rightarrow U(z;\widetilde{\bf{v}}_0), \qquad \text{as $n\rightarrow \infty$},
	\]
and then
	\[
	M^+(z;\widetilde{\bf{v}}_n)\rightarrow M^+(z;\widetilde{\bf{v}}_0)\qquad \text{as $n\rightarrow \infty$},
	\]
where \x{M+} is used. The same holds for $m_{jk}^-(z;\widetilde{\bf{v}}_n)$. Thus we obtain the desired result. 

\fbox{{\rm {\romannumeral2})}}\ : Combining {\rm i)} with \x{green2} and \x{sol-conv}, we have the desired result.

\fbox{{\rm {\romannumeral3})}}\ : Take $x=y$ and consider $\frac{\mathrm{d}}{\mathrm{d}x}G(x,x,z;\widetilde{\bf{v}})$. By \x{green2}, we have
	\begin{equation*}
	\begin{split}
	&\frac{\mathrm{d}}{\mathrm{d}x}G(x,x,z;\widetilde{\textbf{v}})=\sum_{j,k=1}^{2}m_{jk}^+(z;\widetilde{\textbf{v}})(u_j(x,z;\widetilde{\textbf{v}})u_k(x,z;\widetilde{\textbf{v}}))^\prime\\
	&=\sum_{j,k=1}^{2}m_{jk}^+(z;\widetilde{\textbf{v}})\frac{1}{\widetilde{p}(x)}\bigg(\widetilde{p}(x)u_j^\prime(x,z;\widetilde{\textbf{v}})u_k(x,z;\widetilde{\textbf{v}})+u_j(x,z;\widetilde{\textbf{v}})\widetilde{p}(x)u_k^\prime(x,z;\widetilde{\textbf{v}})\bigg).
	\end{split}
	\end{equation*}
For $z\in\mathbb{C}\setminus\sigma(L_{{\bf{v}}})$, it follows from {\rm \romannumeral1)} and \x{sol-conv} that
	\begin{equation*}
	\frac{\mathrm{d}}{\mathrm{d}x}G(x,x,z;\widetilde{\bf{v}}_n)\rightarrow \frac{\mathrm{d}}{\mathrm{d}x}G(x,x,z;\widetilde{\bf{v}}_0),\qquad \text{as $n\rightarrow \infty$},
	\end{equation*}
uniformly on any finite interval of $x \in \mathbb{R}$. The proof is complete. 
\end{proof} 

Based on Lemma \ref{mgg'-con}, we have the following result that play a fundamental role in the proof of gap labelling theorem.

\begin{lem}\label{gg'ap}
For $z\in\mathbb{C}\setminus\sigma(L_{\bf{v}})$, we have
	\[
	G(x,x,z;{\bf{v}}),\quad \frac{\mathrm{d}}{\mathrm{d}x}G(x,x,z;{\bf{v}})\;\in \mathcal{AP}_{\mathcal{M}_{{\bf{v}}}}(\mathbb{R},\mathbb{C}).
	\]
\end{lem}
\begin{proof}
Consider any sequence $\{t_n\}_{n \in \mathbb{N}}\subseteq \mathbb{R}$ such that 	
	\begin{equation} \label{gg'ap1}
	{\bf{v}}\cdot t_n\rightarrow \widetilde{{\bf{v}}}\in \mathrm{E}({\bf{v}}).
	\end{equation}
For any $z\in\mathbb{C}\setminus\sigma(L_{\bf{v}})$, using Lemma \ref{sft-g} {\rm \romannumeral2)} and Lemma \ref{mgg'-con}  {\rm \romannumeral2)} we have
$$G(x+t_n,x+t_n,z;\textbf{v})=G(x,x,z;\textbf{v}\cdot t_n)\rightarrow G(x,x,z;\widetilde{{\bf{v}}}),$$
uniformly on any finite interval of $x\in\mathbb{R}$.

We claim that as $n\rightarrow\infty$, $G(x+t_n,x+t_n,z;\textbf{v})$ converges uniformly for all $x\in\mathbb{R}$. If not,  then there exist two sub-sequences $\left\lbrace k_{1n}\right\rbrace \subset \left\lbrace t_n\right\rbrace$ and $\left\lbrace k_{2n}\right\rbrace \subset \left\lbrace t_n\right\rbrace$ and a sequence $\left\lbrace x_n\right\rbrace \subset \mathbb{R}$ such that 
	\begin{equation}\label{gg'ap2}
	\left|G(x_n+k_{1n},x_n+k_{1n},z;\textbf{v})-G(x_n+k_{2n},x_n+k_{2n},z;\textbf{v})\right| \geq \delta_0,\quad\text{for all $n\in\mathbb{N}$},
	\end{equation}
where $\delta_0>0$.
As $\widetilde{\textbf{v}} \in \mathrm{E}(\textbf{v})$, there exists a sub-sequence $\left\lbrace x_{n_j}\right\rbrace \subseteq\left\lbrace x_n\right\rbrace$  such that
	\[
	 \widetilde{\textbf{v}}\cdot x_{n_j}\to \widetilde{\widetilde{\textbf{v}}}\in \mathrm{E}(\textbf{v}),\qquad\text{as $j\rightarrow\infty$}.
	\]
This implies from \x{gg'ap1} that $\textbf{v}\cdot(x_{n_j}+k_{1n_j})\rightarrow\widetilde{\widetilde{\textbf{v}}}$ and $\textbf{v}\cdot(x_{n_j}+k_{2n_j})\rightarrow\widetilde{\widetilde{\textbf{v}}}$, simultaneously as $j\rightarrow\infty$. Note that
	\begin{equation*}
	\begin{split}
	&\lim\limits_{j\rightarrow\infty}G(x+x_{n_j}+k_{1n_j},x+x_{n_j}+k_{1n_j},z;\textbf{v})=G(x,x,z;\widetilde{\widetilde{\textbf{v}}}),\\
	&\lim\limits_{j\rightarrow\infty}G(x+x_{n_j}+k_{2n_j},x+x_{n_j}+k_{2n_j},z;\textbf{v})=G(x,x,z;\widetilde{\widetilde{\textbf{v}}}),
	\end{split}
	\end{equation*}
uniformly on any finite interval of $x \in \mathbb{R}$. If $x=0$, this is contradict with \x{gg'ap2}. 
Hence the proof of the claim is complete. By Lemma \ref{mv-eq}, we conclude that $G(x,x,z;\textbf{v})\in\mathcal{AP}_{\mathcal{M}_{{\bf{v}}}}(\mathbb{R},\mathbb{C})$. 

The argument for $\frac{\mathrm{d}}{\mathrm{d}x}G(x,x,z;{\bf{v}}) \in \mathcal{AP}_{\mathcal{M}_{{\bf{v}}}}(\mathbb{R},\mathbb{C})$ is analogous to that for $G(x,x,z;\textbf{v}) \in \mathcal{AP}_{\mathcal{M}_{{\bf{v}}}}(\mathbb{R},\mathbb{C})$. The proof is complete. \end{proof}

\section{Rotation number}\label{sec4}

\setcounter{equation}{0}

\subsection{Pr\"{u}fer angle}

Let $\textbf{v}$ in \x{v} be fixed. As in Subsection \ref{ss-slh}, we consider a family of operators $\tau_{\widetilde{\textbf{v}}}$ in \x{sl-ex} as $\widetilde{\textbf{v}}$ varies over $\mathrm{E}(\textbf{v})$. Introduce the so-called \textsl{Pr\"{u}fer transformation} as
	\begin{equation} \label{pufer}
	\widetilde{p}(x)\phi'(x)+ \mathrm{i}\, \phi(x)= r(x)\,\mathrm{e}^{\mathrm{i}\, \theta(x)},
	\end{equation}
where $\phi(x)$ is any non-trivial solution of 
\begin{equation} \label{tau-lam}
	\tau_{\widetilde{\textbf{v}}}\phi=\lambda\phi, \qquad \lambda \in \R.
\end{equation}
Then the Pr\"{u}fer angle $\theta(x)$ satisfies the differential equation
	\begin{equation}\label{ftheta}
	\theta^{\prime}(x)=\frac{1}{\widetilde{p}(x)}\cos^2\theta(x)+\big(\lambda \widetilde{w}(x)-\widetilde{q}(x)\big) \sin^2\theta(x).
	\end{equation}

The following basic properties of the solution $\theta(x)$ of (\ref{ftheta}) are easily deduced from the classical ODE theory.

\begin{lem}\label{exist}
	For any $\mathtt{\Theta}\in\mathbb{R}$, there exists a unique solution
	of {\rm (\ref{ftheta})} with $\theta_\lambda(0,\mathtt{\Theta};\widetilde{\bf{v}})=\mathtt{\Theta}$. Moreover, the solution is continuously defined for all $x\in\mathbb{R}$. We denote it by $\theta(x):=\theta_\lambda(x,\mathtt{\Theta};\widetilde{\bf{v}})$ for $x\in\mathbb{R}$.
\end{lem}

It follows from \x{v} and Lemma \ref{lm-1/f} \romannumeral1) that
\begin{lem}\label{th-la}
	When $x \in \R_+$, $\mathtt{\Theta}\in\mathbb{R}$ and $\widetilde{\bf{v}}\in \mathrm{E}(\bf{v})$ are fixed, $\theta_\lambda(x,\mathtt{\Theta};\widetilde{\bf{v}})$ is strictly increasing and continuous with respect to\/ $\lambda$.
\end{lem}

We now state some results about $\theta_\lambda(x,\mathtt{\Theta};\widetilde{\bf{v}})$ that we need, which are not surprising. For the sake of completeness, we nevertheless decide to sketch their proofs.

\begin{lem}\label{lm-theta}  Let $\lambda \in \R$ be fixed.
\begin{itemize}
	\item[{\rm \romannumeral1):}] When $\mathtt{\Theta}\in\mathbb{R}$, $k\in\mathbb{Z}$ and $\widetilde{\bf{v}}\in \mathrm{E}(\bf{v})$ are fixed, one has
		\[
		\theta_\lambda(x, \mathtt{\Theta}+2k\pi; \widetilde{\bf{v}})=\theta_\lambda(x, \mathtt{\Theta}; \widetilde{\bf{v}})+2k\pi\qquad \text{for all~}x\in\mathbb{R}.
		\]
	\item[{\rm \romannumeral2):}] When $\mathtt{\Theta}\in\mathbb{R}$ and $\widetilde{\bf{v}}\in \mathrm{E}(\bf{v})$ are fixed, one has
		\[
		\theta_\lambda(x_1+ x_2, \mathtt{\Theta}; \widetilde{\bf{v}})=\theta_\lambda(x_1,\theta_\lambda(x_2, \mathtt{\Theta}; \widetilde{\bf{v}});\widetilde{\bf{v}}\cdot x_2)\qquad \mbox{for all~} x_1, x_2 \in\mathbb{R}.
		\]
	\item[{\rm \romannumeral3):}] When $x_0 \in\mathbb{R}$ is fixed, $\theta_\lambda(x_0, \mathtt{\Theta}; \widetilde{\bf{v}}):\mathbb{R}\times\mathrm{E}(\bf{v})\rightarrow \mathbb{R}$ is Lipschitz continuous.
	\item[{\rm \romannumeral4):}] When $x \in\mathbb{R}$ and $\widetilde{\bf{v}}\in E(\bf{v})$ are fixed, $\theta_\lambda(x,\mathtt{\Theta};\widetilde{\bf{v}}):\mathbb{R}\rightarrow \mathbb{R}$ is a strictly increasing self-homeomorphism.
	\end{itemize}
\end{lem}
\begin{proof}
\fbox{{\rm {\romannumeral1})}}\ :  Due to the $2\pi$-periodicity in $\theta$ of the vector field of (\ref{ftheta}) and by Lemma \ref{exist}, we have the desired result.

\fbox{{\rm {\romannumeral2})}}\ : Denote $\bar{\theta}_{1}(x):=\theta_\lambda(x,\theta_\lambda(x_2, \mathtt{\Theta}; \widetilde{\textbf{v}});\widetilde{\textbf{v}}\cdot x_2)$ and $\bar{\theta}_{2}(x):=\theta_\lambda(x+ x_2, \mathtt{\Theta}; \widetilde{\textbf{v}})$. Then both $\bar{\theta}_{1}(x)$
and $\bar{\theta}_{2}(x)$ satisfy the following initial problem
	\begin{equation*}
	\begin{cases}
		\theta^{\prime}(x)=\frac{1}{\widetilde{p}(x+x_2)}\cos^2\theta(x)+(\lambda \widetilde{w}(x+x_2)-\widetilde{q}(x+x_2))\sin^2\theta(x),\\
		\theta(0)=\theta_\lambda(x_2, \mathtt{\Theta};\widetilde{\textbf{v}}).
	\end{cases}
	\end{equation*}
Due to Lemma \ref{exist}, we conclude that $\bar{\theta}_{1}(x)=\bar{\theta}_{2}(x)$ for all $x\in\mathbb{R}$. Taking $x=x_1$, we have the
desired result.

\fbox{{\rm {\romannumeral3})}}\ : For $i=1,2$, let $\theta_{i}(x):=\theta_\lambda(x, \mathtt{\Theta}_i;\widetilde{\textbf{v}}_i)$ be the solutions of (\ref{ftheta}) with different $\widetilde{\textbf{v}}_i$ and $\mathtt{\Theta}_i$. Then they satisfy 
	\begin{equation}\label{lm-theta1}
	\theta_{i}(x)=\mathtt{\Theta}_i+\int_0^x\left( \frac{1}{\widetilde{p}_i(s)}\cos^2\theta_{i}(s)+(\lambda \widetilde{w}_i(s)-\widetilde{q}_i(s))\sin^2\theta_{ i}(s)\right)\mathrm{d}s,\quad i=1,2.
	\end{equation}
Note that
	\begin{equation} \label{lm-theta2}
	\begin{split}
		&\left(\frac{1}{\widetilde{p}_2(s)}\cos^2\theta_{2}(s)+\big(\lambda \widetilde{w}_2(s)-\widetilde{q}_2(s)\big)\sin^2\theta_{ 2}(s)\right)\\
		&-\left(\frac{1}{\widetilde{p}_1(s)}\cos^2\theta_{1}(s)+\big(\lambda \widetilde{w}_1(s)-\widetilde{q}_1(s)\big)\sin^2\theta_{1}(s)\right)\\
		=&\Big(\frac{1}{\widetilde{p}_2(s)}-\frac{1}{\widetilde{p}_1(s)}\Big)\cos^2\theta_{2}(s)+\frac{1}{\widetilde{p}_1(s)}\big(\cos^2\theta_{ 2}(s)-\cos^2\theta_{ 1}(s)\big)\\
		&+\big(\lambda \widetilde{w}_2(s)-\widetilde{q}_2(s)-\lambda \widetilde{w}_1(s)+\widetilde{q}_1(s)\big)\sin^2\theta_{2}(s) \\
		&+\big(\lambda \widetilde{w}_1(s)-\widetilde{q}_1(s)\big)\big(\sin^2\theta_{ 2}(s)-\sin^2\theta_{1}(s)\big). 
	\end{split}
	\end{equation}
When $s$ is fixed, we can regard $\theta_{1}(s)$ and $\theta_{2}(s)$ as two real numbers. By the Mean Value Theorem,
there exist $\zeta(s)$ and $\eta(s)$ which belong to the interval with endpoints $\theta_{1}(s)$ and $\theta_{2}(s)$ such that
	\begin{equation*}\label{3.5}
	\cos\theta_{2}(s)-\cos\theta_{1}(s)=-\sin\zeta(s)(\theta_{ 2}(s)-\theta_{ 1}(s)),
	\end{equation*}
	\begin{equation*}\label{3.6}
	\sin\theta_{2}(s)-\sin\theta_{1}(s)=\cos\eta(s)(\theta_{ 2}(s)-\theta_{ 1}(s)).
	\end{equation*}
Without loss of generality, we assume that $x_0>0$. Denote 
	\[
	D(x):=\theta_{2}(x)-\theta_{1}(x), \qq x \in [0,x_0],
	\]
	\[
	B(s):=\left(\frac{1}{\widetilde{p}_2(s)}-\frac{1}{\widetilde{p}_1(s)}\right)\cos^2\theta_{2}(s)+\big(\lambda \widetilde{w}_2(s)-\widetilde{q}_2(s)-\lambda \widetilde{w}_1(s)+\widetilde{q}_1(s)\big)\sin^2\theta_{2}(s),
	\]
	\[
	A(s):=-\frac{1}{\widetilde{p}_1(s)}\sin\zeta(s)\big(\cos\theta_{ 2}(s)+\cos\theta_{1}(s)\big)+\big(\lambda \widetilde{w}_1(s)-\widetilde{q}_1(s)\big)\cos\eta(s)\big(\sin\theta_{ 2}(s)+\sin\theta_{1}(s)\big),
	\]
for $s \in[0,x]$.
Combining $A(s)$ and $B(s)$ with (\ref{lm-theta1}) and (\ref{lm-theta2}), we have
	\[
	D(x)=(\mathtt{\Theta}_2-\mathtt{\Theta}_1)+\int_0^x\big(A(s)D(s)+B(s)\big) \mathrm{d}s, \qquad x \in [0,x_0].
	\]
Then
	\[
	|D(x)|\leq|\mathtt{\Theta}_2-\mathtt{\Theta}_1|+\int_0^x\big(|A(s)D(s)|+|B(s)|\big) \mathrm{d}s, \qquad x \in [0,x_0].
	\]
Denote $\disp C(x):=|\mathtt{\Theta}_2-\mathtt{\Theta}_1|+\int_0^x|B(s)| \mathrm{d}s$. Then we have
	\begin{equation*}
	\begin{split}
		C(x)\leq& |\mathtt{\Theta}_2-\mathtt{\Theta}_1|+x_0\left( \Big\|\frac{1}{\widetilde{p}_2}-\frac{1}{\widetilde{p}_1}\Big\|_{\infty}+\|\widetilde{q}_2-\widetilde{q}_1\|_{\infty}+|\lambda|\|\widetilde{w}_2-\widetilde{w}_1\|_{\infty}\right) \\
		\leq& |\mathtt{\Theta}_2-\mathtt{\Theta}_1|+3x_0\max\{1,|\lambda|\}\|\widetilde{\textbf{v}}_2-\widetilde{\textbf{v}}_1\|_{\infty},
	\end{split}
	\end{equation*}
where \x{dist-v} is used. By the generalized Gronwall inequality \cite[Corollary 6.6]{ode}, we obtain
	\begin{equation}\label{lm-theta3}
	|D(x)|\leq C(x)+\int_0^x|A(s)||C(s)| \mathrm{e}^{\int_{s}^{x}|A(t)|\mathrm{d}t} \mathrm{d}s, \qquad x \in [0,x_0].
	\end{equation}
Taking $x=x_0$ in (\ref{lm-theta3}), we have
	\begin{equation*}
	\begin{split}
		&|D(x_0)|\\
		\leq&\big(|\mathtt{\Theta}_2-\mathtt{\Theta}_1|+3x_0\max\{1,|\lambda|\}\|\widetilde{\textbf{v}}_2-\widetilde{\textbf{v}}_1\|_{\infty}\big)\\
		&\cdot \Bigg(1+2\exp\left(\disp 2\int_0^{x_0} \Big(\Big|\frac{1}{\widetilde{p}_1(t)}\Big|+|\widetilde{q}_1(t)|+|\lambda||\widetilde{w}_1(t)|\Big)\mathrm{d}t\right)\\
		&\cdot \int_0^{x_0} \Big(\Big|\frac{1}{\widetilde{p}_1(s)}\Big|+|\widetilde{q}_1(s)|+|\lambda||\widetilde{w}_1(s)|\Big)\mathrm{d}s \Bigg)\\
		\leq&\big(|\mathtt{\Theta}_2-\mathtt{\Theta}_1|+3x_0\max\{1,|\lambda|\}\|\widetilde{\textbf{v}}_2-\widetilde{\textbf{v}}_1\|_{\infty}\big)\cdot \big(1+6\mathrm{e}^{6x_0\|\widetilde{\textbf{v}}_1\|_{\infty}}x_0 \|\widetilde{\textbf{v}}_1\|_{\infty} \big),
	\end{split}
	\end{equation*}
where Lemma \ref{lm-1/f} and \x{dist-v} are used. Since $\|\widetilde{\textbf{v}}_1\|_{\infty}=\|\textbf{v}\|_{\infty}$ is fixed, we have the desired result.

\fbox{{\rm {\romannumeral4})}}\ : Due to Lemma \ref{exist}, we know that for any fixed $\widetilde{\bf{v}}\in E(\bf{v})$ and $x\in\R$, $\theta_{\lambda}(x, \mathtt{\Theta}; \widetilde{\textbf{v}}):\mathbb{R}\rightarrow \mathbb{R}$ is strictly increasing. By {\rm {\romannumeral2})}, we
have
	\[
	\theta_{\lambda}\big(x,\theta_{\lambda}(-x,\mathtt{\Theta};\widetilde{\textbf{v}}\cdot x); \widetilde{\textbf{v}}\big)=\mathtt{\Theta}=\theta_{\lambda}\big(-x,\theta_{\lambda}(x, \mathtt{\Theta}; \widetilde{\textbf{v}});\widetilde{\textbf{v}}\cdot x\big).
	\]
This implies that the inverse of $\theta_{\lambda}(x, \mathtt{\Theta}; \widetilde{\textbf{v}})$ is $\theta_{\lambda}(-x,\mathtt{\Theta};\widetilde{\textbf{v}}\cdot x)$. Combining this with {\rm {\romannumeral3})}, we complete the proof.\end{proof}

\subsection{Reduction to skew-products}

Let $\mathbb{S}_{2\pi}:=\mathbb{R}\slash 2\pi\mathbb{Z}$. Introduce the product space $\mathrm{Z}:=\mathbb{S}_{2\pi} \times \mathrm{E}(\textbf{v})$ with the distance
	\[
	d\big((\vartheta_1,\widetilde{\textbf{v}}_1),(\vartheta_2,\widetilde{\textbf{v}}_2)\big):=\max\left\lbrace |\vartheta_1-\vartheta_2|_{2\pi}, \|\widetilde{\textbf{v}}_1-\widetilde{\textbf{v}}_2\|_{\infty} \right\rbrace, 
	\]
where $(\vartheta_i,\widetilde{\textbf{v}}_i)\in\mathrm{Z},\, i=1,2$, and $|\vartheta_1-\vartheta_2|_{2\pi}:=|\vartheta_1-\vartheta_2| \mod 2\pi$. Lemma \ref{E-cpt} implies that $\mathrm{Z}$ is a compact metric space. 

For each $t\in\mathbb{R}$, the skew-product transformation $\Psi_\lambda^t$ on $\mathrm{Z}$ can be defined by
	\begin{equation}\label{skew}
	\Psi_\lambda^t(\vartheta,\widetilde{\textbf{v}}):=\big(\theta_{\lambda}(t,\mathtt{\Theta};\widetilde{\textbf{v}})\mod 2\pi,\widetilde{\textbf{v}}\cdot t \big) ,\quad (\vartheta,\widetilde{\textbf{v}})\in \mathrm{Z},
	\end{equation}
where $\mathtt{\Theta} \in \mathbb{R}$ satisfies $\vartheta=\mathtt{\Theta}\mod 2\pi$. By Lemma \ref{lm-theta} {\rm {\romannumeral1})}, $\Psi_\lambda^t$ is well-defined.
Moreover, we have
\begin{lem}\label{lm-ds}	
	$\big\lbrace\Psi_\lambda^t\big\rbrace_{t\in\mathbb{R}}$ defined by {\rm (\ref{skew})} is a continuous skew-product on the compact space $\mathrm{Z}$.
\end{lem}
\begin{proof}
The continuity of $\Psi_\lambda^t$ can be deduced from \x{sft-is} and Lemma \ref{lm-theta} {\rm {\romannumeral3})}. We only need to prove that
	\[
	\Psi_\lambda^{t_1+t_2}=\Psi_\lambda^{t_1}\circ\Psi_\lambda^{t_2}\qquad \text{for all}\; t_1, t_2\in \mathbb{R}.
	\]
Indeed, for $(\vartheta,\widetilde{\textbf{v}})\in \mathrm{Z}$, there exists a $\mathtt{\Theta}\in\mathbb{R}$ such that $\vartheta=\mathtt{\Theta}\mod 2\pi$. Then we have
	\begin{equation*}
	\begin{split}
		\Psi_\lambda^{t_1}\circ\Psi_\lambda^{t_2}(\vartheta,\widetilde{\textbf{v}})
		&=\Psi_\lambda^{t_1}\big(\theta_{\lambda}(t_2,\mathtt{\Theta};\widetilde{\textbf{v}})\mod 2\pi, \widetilde{\textbf{v}}\cdot t_2\big)\\
		&=\left(\theta_{\lambda}\big(t_1,\theta_{\lambda}(t_2,\mathtt{\Theta};\widetilde{\textbf{v}});\widetilde{\textbf{v}}\cdot t_2\big)\mod 2\pi, \widetilde{\textbf{v}}\cdot(t_2+t_1)\right)\\
		&=\big(\theta_{\lambda}(t_1+t_2,\mathtt{\Theta};\widetilde{\textbf{v}})\mod 2\pi,\widetilde{\textbf{v}}\cdot(t_1+t_2)\big)\\
		&=\Psi_\lambda^{t_1+t_2}(\vartheta,\widetilde{\textbf{v}}),
	\end{split}
	\end{equation*}
where Lemma \ref{lm-theta} {\rm {\romannumeral2})} is used. The proof is complete.\end{proof}

Introduce an observation function $f_\lambda: \mathrm{Z} \to \R$ by
	\begin{equation} \label{df-f}
	f_\lambda(\vartheta,\widetilde{\textbf{v}}):=\frac{1}{\widetilde{p}(0)}\cos^2\vartheta+\big(\lambda \widetilde{w}(0)-\widetilde{q}(0)\big) \sin^2\vartheta.
	\end{equation}
By \x{dist-v}, it is obvious to verify that 
\begin{lem} \label{lm-f} 
	$f_\lambda(\vartheta,\widetilde{\bf{v}})$ is Lipschitz continuous on $\mathrm{Z}$.
\end{lem}

Based on the above construction, we can relate the existence of the rotation number to the convergence of the following time average under the skew-product $\big\{\Psi_\lambda^t\big\}_{t \in \R}$.

\begin{lem} \label{eqs} 
For any $(\vartheta,\widetilde{\bf{v}})\in \mathrm{Z}$ and $x \in \R$, one has 
	\[
	\frac{\theta_\lambda(x,\mathtt{\Theta};\widetilde{\bf{v}})-\mathtt{\Theta}}{x}=\frac{\theta_\lambda(x,\mathtt{\Theta};\widetilde{\bf{v}})-\theta_\lambda(0,\mathtt{\Theta};\widetilde{\bf{v}})}{x}=\frac{1}{x}\int_{0}^{x}f_\lambda\big(\Psi_\lambda^s(\vartheta,\widetilde{\bf{v}})\big)\mathrm{d}s,
	\]
where $\mathtt{\Theta} \in \mathbb{R}$ satisfies $\vartheta=\mathtt{\Theta}\mod 2\pi$.
\end{lem}

\begin{proof}
For $s \in \R$, by \x{skew}, \x{df-f} and \x{sft-f}, we have
	\begin{equation*}
	\begin{split}
		f_\lambda\big(\Psi_\lambda^s(\vartheta,\widetilde{\bf{v}})\big)
		&=\frac{1}{\widetilde{p}\cdot s(0)}\cos^2\theta_{\lambda}(s,\mathtt{\Theta};\widetilde{\textbf{v}})+\big(\lambda \widetilde{w}\cdot s(0)-\widetilde{q}\cdot s(0)\big) \sin^2\theta_{\lambda}(s,\mathtt{\Theta};\widetilde{\textbf{v}})\\
		&=\frac{1}{\widetilde{p}(s)}\cos^2\theta_{\lambda}(s,\mathtt{\Theta};\widetilde{\textbf{v}})+\big(\lambda \widetilde{w}(s)-\widetilde{q}(s)\big) \sin^2\theta_{\lambda}(s,\mathtt{\Theta};\widetilde{\textbf{v}}).
	\end{split}
	\end{equation*}
Combining this with \x{ftheta}, we complete the proof. \end{proof}

According to the Krylov-Bogoliubov theorem \cite[Corollary 6.9.1]{Wa82} and Lemma \ref{lm-ds}, the flow $\big\lbrace \Phi_\lambda^t\big\rbrace_{t \in \R}$ possesses at least one invariant probability measure on the compact metric space $\mathrm{Z}$, denoted by $\nu_\lambda$. In general, $\big\lbrace \Phi_\lambda^t\big\rbrace_{t \in \R}$ is not uniquely ergodic. We then establish the following relationship between any invariant measure $\nu_\lambda$ on $\mathrm{Z}$ and the Haar measure $\mu_{\mathrm{E}(\textbf{v})}$ on $\mathrm{E}(\textbf{v})$.

\begin{lem}\label{proj-me}
	Let $\Pi: \mathrm{Z} \rightarrow \mathrm{E}(\bf{v})$ be the projection.
	Then $\mu_{\mathrm{E}(\bf{v})}=\nu_\lambda\circ\Pi^{-1}$ for each invariant probability measure $\nu_\lambda$.
\end{lem}
\begin{proof}
By Lemma \ref{atg} {\rm {\romannumeral4})}, we only need to verify that $\nu_\lambda\circ\Pi^{-1}$ is invariant on $\mathrm{E}(\textbf{v})$ under the shift \x{sft-v}. Let $\mathrm{B} \subseteq \mathrm{E}(\textbf{v})$ be any Borel subset, and $\mathrm{B}\cdot(-t):=\big\lbrace \widetilde{\textbf{v}}\cdot(-t): \widetilde{\textbf{v}}\in\mathrm{B}\big\rbrace$, for any $t\in\R$. By \x{skew} and Lemma \ref{lm-theta} {\rm {\romannumeral4})}, we have
	\begin{equation} \label{pro-mea1}
	\big(\Psi_\lambda^t\big)^{-1}(\mathbb{S}_{2\pi} \times \mathrm{B})=\mathbb{S}_{2\pi} \times\mathrm{B}\cdot(-t).
	\end{equation}
Since $\nu_\lambda$ is invariant, we obtain 
	\begin{align*}
	\nu_\lambda\circ\Pi^{-1}(\mathrm{B}) &=\nu_\lambda\left(\mathbb{S}_{2\pi} \times \mathrm{B}\right)=\nu_\lambda\left(\big(\Psi_\lambda^t\big)^{-1}(\mathbb{S}_{2\pi} \times \mathrm{B})\right)\\
	&=\nu_\lambda\big(\mathbb{S}_{2\pi} \times \mathrm{B}\cdot(-t)\big)=\nu_\lambda\circ\Pi^{-1}\big(\mathrm{B}\cdot(-t)\big),
	\end{align*}
where \x{pro-mea1} is used. Since $t$ is arbitrary, the proof is complete. \end{proof}

\subsection{Rotation number}
In this subsection, we first prove the existence of the rotation number and then derive its fundamental properties. Recalling Lemma \ref{eqs}, we introduce the following two notations.
	\begin{equation*}
	\bar{f}_\lambda(\mathtt{\Theta},\widetilde{\textbf{v}}):=\lim\limits_{x\rightarrow +\infty}\frac{\theta_\lambda(x,\mathtt{\Theta};\widetilde{\textbf{v}})-\mathtt{\Theta}}{x}\qquad \mbox{for~}(\mathtt{\Theta},\widetilde{\textbf{v}})\in\R \times \mathrm{E}(\textbf{v}),
	\end{equation*} 
and
	\begin{equation}\label{hatf}
	\mathring{f}_\lambda(\vartheta,\widetilde{\textbf{v}}):=\lim\limits_{x\rightarrow +\infty}\frac{1}{x}\int_{0}^{x}f_\lambda\big(\Psi_\lambda^s(\vartheta,\widetilde{\textbf{v}})\big)\mathrm{d}s \qquad \mbox{for~} (\vartheta,\widetilde{\textbf{v}}) \in \mathrm{Z},
	\end{equation}
if this limit exists. Then we have
\begin{lem}\label{ind}
	If $\mathring{f}_\lambda(\vartheta_0,\widetilde{\bf{v}}_0)$ exists for some $(\vartheta_0,\widetilde{\bf{v}}_0) \in \mathrm{Z}$, then $\mathring{f}_\lambda(\vartheta,\widetilde{\bf{v}}_0)$ exists for all $\vartheta \in \mathbb{S}_{2\pi}$ and is independent of the choice of $\vartheta$.
\end{lem}

\begin{proof}
By Lemma \ref{eqs}, we know that there exists a $\mathtt{\Theta}_0 \in \mathbb{R}$ satisfying $\vartheta_0=\mathtt{\Theta}_0\mod 2\pi$ such that $\bar{f}_\lambda(\mathtt{\Theta}_0,\widetilde{\textbf{v}}_0)$ exists. By Lemma \ref{lm-theta} {\rm {\romannumeral1})}, we have that $\bar{f}_\lambda(\mathtt{\Theta}_0+2k\pi,\widetilde{\textbf{v}}_0)$ exist for all $k\in\mathbb{Z}$. Moreover, for each $\mathtt{\Theta}\in\mathbb{R}$, there exists a $k_{\mathtt{\Theta}}\in\mathbb{Z}$ such that
$\mathtt{\Theta}_0+2k_{\mathtt{\Theta}}\pi\leq\mathtt{\Theta}<\mathtt{\Theta}_0+2(k_{\mathtt{\Theta}}+1)\pi$. By Lemma \ref{lm-theta} {\rm {\romannumeral1})} and {\rm {\romannumeral4})}, we have that 
	\begin{equation*}
	\begin{split}
		\theta_\lambda(x,\mathtt{\Theta}_0+2k_{\mathtt{\Theta}}\pi;\widetilde{\textbf{v}})&\leq \theta_\lambda(x,\mathtt{\Theta};\widetilde{\textbf{v}})<\theta_\lambda(x,\mathtt{\Theta}_0+2(k_{\mathtt{\Theta}}+1)\pi;\widetilde{\textbf{v}})\\
		&=\theta_\lambda(x,\mathtt{\Theta}_0+2k_{\mathtt{\Theta}}\pi;\widetilde{\textbf{v}})+2\pi \qquad \text{for all~}x>0.
	\end{split}
	\end{equation*}
This implies that $\bar{f}_\lambda(\mathtt{\Theta},\widetilde{\textbf{v}}_0)$ exists as well for all $\mathtt{\Theta}\in\mathbb{R}$ and is independent of the choice of $\mathtt{\Theta}$. By Lemma \ref{eqs} again, we obtain the desired result.
\end{proof}

We present the following uniformly ergodic theorem by Johnson and Moser, which will be used to establish the existence of the rotation number. For its discrete version, see \cite{He83}.

\begin{lem}[{\cite[Lemma 4.4]{JM1982}}]\label{jm}
Let $\big\lbrace\Psi^t\big\rbrace_{t\in\mathbb{R}}$ be a continuous dynamical system on a compact metric space $X$. Then for any continuous function $g:X \to \K$ satisfying
	\[
	\int_X g(x)\mathrm{d}\nu=0
	\]
for all invariant Borel probability measures $\nu$ of $\big\lbrace \Psi^t\big\rbrace$, one has
	\[
	\lim\limits_{t\rightarrow+\infty}\frac{1}{t}\int_{0}^{t}g\big(\Psi^s(x)\big)\mathrm{d}s=0
	\]
uniformly for all $x\in X$.
\end{lem}

\begin{proof}[Proof of Theorem \ref{thm-rn}.]
Using the argument of Johnson and Moser \cite{JM1982}, we consider the skew-product $\big\{\Psi_\lambda^t\big\}_{t\in \R}$ on $\mathrm{Z}$. The crucial difference is that we adopt the joint hull $\mathrm{E}(\textbf{v})$.

Step 1: Let $\nu_\lambda$ be any invariant probability measure on the compact metric space $\mathrm{Z}$. By the Birkhoff ergodic theorem, there exists a Borel set $\mathrm{Z}_0\subseteq\mathrm{Z}$ that depends on the choice of $\nu_\lambda$ such that $\nu_\lambda(\mathrm{Z}_0)=1$ and $\mathring{f}_\lambda(\vartheta,\widetilde{\textbf{v}})$ in (\ref{hatf})
exists for all $(\vartheta,\widetilde{\textbf{v}})\in\mathrm{Z}_0$. By Lemma \ref{ind}, $\mathrm{Z}_0$ can be expressed as 
	\begin{equation*}
	\mathrm{Z}_0=\mathbb{S}_{2\pi}\times \mathrm{E}_0.
	\end{equation*}
Combining this with Lemma \ref{proj-me}, we obtain 
	\begin{equation} \label{thm-rn0}
	\mu_{\mathrm{E}(\textbf{v})}(\mathrm{E}_0)=\nu_\lambda\circ\Pi^{-1}(\mathrm{E}_0)=\nu_\lambda(\mathbb{S}_{2\pi} \times \mathrm{E}_0)=1.
	\end{equation}
Meanwhile, $\mathring{f}_\lambda(\vartheta,\widetilde{\textbf{v}})\in \mathcal{L}^1(\mathrm{Z},\mathrm{d}\nu_\lambda)$ is invariant under $\big\{\Psi_\lambda^t\big\}$ and satisfies
	\begin{equation}\label{thm-rn1}
	\int_\mathrm{Z}\mathring{f}_\lambda(\vartheta,\widetilde{\textbf{v}})\mathrm{d}\nu_\lambda=\int_\mathrm{Z}f_\lambda(\vartheta,\widetilde{\textbf{v}})\mathrm{d}\nu_\lambda=:\rho(\lambda,\textbf{v}),
	\end{equation}
where $\rho(\lambda,\textbf{v}) \in \R$ is a constant that depends on $\lambda$ and $\textbf{v}$. 

Step 2: By Lemma \ref{ind} and (\ref{thm-rn0}), $\mathring{f}_\lambda(\vartheta,\widetilde{\textbf{v}})$ can be regarded as a function on $\mathrm{E}(\textbf{v})$. We still denote
	\begin{equation} \label{thm-rn2}
	\mathring{f}_\lambda(\widetilde{\textbf{v}}):=\mathring{f}_\lambda(\vartheta,\widetilde{\textbf{v}}) \quad \mbox{for~} \mu_{\mathrm{E}(\textbf{v})}\mbox{-almost~everywhere~}\widetilde{\textbf{v}} \in \mathrm{E}(\textbf{v}).
	\end{equation}
By \x{skew}, $\mathring{f}_\lambda(\widetilde{\textbf{v}})$ is invariant under the shift (\ref{sft-v}) on $\mathrm{E}(\textbf{v})$. By Lemma \ref{atg} {\rm {\romannumeral4})}, we know that the Haar measure $\mu_{\mathrm{E}(\textbf{v})}$ is ergodic. This implies that $\mathring{f}_\lambda(\widetilde{\textbf{v}})$ is constant for $\mu_{\mathrm{E}(\textbf{v})}$-almost every points. Combining this with \x{thm-rn1} and \x{thm-rn2}, we have
	\begin{equation} \label{thm-rn3}
	\mathring{f}_\lambda(\widetilde{\textbf{v}})=\rho(\lambda,\textbf{v})\qquad \mbox{for~all~}\widetilde{\textbf{v}} \in \mathrm{E}_1 \subseteq \mathrm{E}(\textbf{v}),
	\end{equation}
where $\mathrm{E}_1$ is a Borel set of $\mu_{\mathrm{E}(\textbf{v})}$-full measure.

Step 3: We claim that $\rho(\lambda,\textbf{v})$ is independent of the choice of the measure $\nu_\lambda$. To this end, assume that there exists another measure $\tilde{\nu}_\lambda$. As in Step 1 and Step 2, there exists another Borel set $\tilde{\mathrm{E}}_1 \subseteq \mathrm{E}(\textbf{v})$ of $\mu_{\mathrm{E}(\textbf{v})}$-full measure and a corresponding constant $\tilde{\rho}(\lambda,\textbf{v})$. We then obtain $\mu_{\mathrm{E}(\textbf{v})}(\mathrm{E}_1 \cap \tilde{\mathrm{E}}_1)=1$. This implies that $\mathrm{E}_1 \cap \tilde{\mathrm{E}}_1 \neq\emptyset$. Choosing a point $\widetilde{\textbf{v}}_0\in \mathrm{E}_1 \cap \tilde{\mathrm{E}}_1$, we have
	\[
	\rho(\lambda,\textbf{v})=\mathring{f}_\lambda(\vartheta,\widetilde{\textbf{v}}_0)=\tilde{\rho}(\lambda,\textbf{v})\qquad \mbox{for~all~}\vartheta\in\mathbb{S}_{2\pi},
	\]
where (\ref{hatf}), (\ref{thm-rn2}) and (\ref{thm-rn3}) are used. 

Step 4: Denote the function $g_\lambda$ by 	
	\[
	g_\lambda(\vartheta,\widetilde{\textbf{v}}):=f_\lambda(\vartheta,\widetilde{\textbf{v}})-\rho(\lambda,\textbf{v}).
	\]
It follows from Lemma \ref{lm-ds} and Lemma \ref{lm-f} that $\big\lbrace\Psi_\lambda^t\big\rbrace$ and  $g_\lambda$ satisfy all the requirements of Lemma \ref{jm}. As $x\rightarrow +\infty$, we then obtain
	\begin{equation}\label{thm-rn4}
	\frac{1}{x}\int_{0}^{x}g_\lambda\big(\Psi_\lambda^s(\vartheta,\widetilde{\textbf{v}})\big)\mathrm{d}s=\frac{1}{x}\int_{0}^{x}f_\lambda\big(\Psi_\lambda^s(\vartheta,\widetilde{\textbf{v}})\big)\mathrm{d}s-\rho(\lambda,\textbf{v})\rightarrow 0
	\end{equation}
uniformly for all $(\vartheta,\widetilde{\textbf{v}})\in\mathrm{Z}$. This implies that $\mathring{f}_\lambda(\vartheta,\widetilde{\textbf{v}})$ exists for all $(\vartheta,\widetilde{\textbf{v}})\in\mathrm{Z}$.

Finally, taking $\widetilde{\textbf{v}}=\textbf{v}$ in (\ref{thm-rn4}), we complete the proof.
\end{proof}

\begin{rem} \label{re-rn}
	{\rm \romannumeral1): For each $\widetilde{\textbf{v}}\in\mathrm{E}(\textbf{v})$, denote the rotation number of \x{tau-lam} by $\rho(\lambda,\widetilde{\textbf{v}})$. By \rm{\x{thm-rn4}}, we have $\rho(\lambda,\widetilde{\textbf{v}})\equiv \rho(\lambda,\textbf{v})$.\\[1mm]
	\romannumeral2): According to \x{thm-rn1}, the rotation number admits an ergodic representation in terms of the space average as follows.
		\[
		\rho(\lambda,\textbf{v})=\int_\mathrm{Z}f_\lambda(\vartheta,\widetilde{\textbf{v}})\mathrm{d}\nu_\lambda,
		\]
	where $\nu_\lambda$ is any invariant Borel probability measure of $\big\{\Psi^t_\lambda\big\}$. Note that $\rho(\lambda,\textbf{v})$ is independent of the choice of $\nu_\lambda$.
	}
\end{rem}

\begin{lem}\label{rn-cont}
	Let $\bf{v}$ be fixed. Then $\rho(\lambda,{\bf{v}})$ is continuous with respect to\/ $\lambda$.
\end{lem}
\begin{proof}
It is equivalent to show that for any sequence $\lambda_j \to \lambda_0 \in \R$, the following holds
	\begin{equation} \label{rn-cont1}
	\rho(\lambda_j,{\bf{v}}) \to \rho(\lambda_0,{\bf{v}})\qquad \mbox{as~} j \to +\infty.
	\end{equation}
By Remark \ref{re-rn} \romannumeral2), we have 
	\begin{equation} \label{rn-cont1.1}
	\rho(\lambda_j,\textbf{v})=\int_\mathrm{Z}f_{\lambda_j}(\vartheta,\widetilde{\textbf{v}})\mathrm{d}\nu_{\lambda_j} \qquad \mbox{for~any~} j\in\N,
	\end{equation}
where $f_{\lambda_j}(\vartheta,\widetilde{\textbf{v}})$ is given by (\ref{df-f}), and $\nu_{\lambda_j}$ is any invariant Borel probability measure of $\big\{\Psi^t_{\lambda_j}\big\}$. Since $\mathrm{Z}$ is a compact metric space, by \cite[Theorem 6.5]{Wa82}, we may assume that there exists a Borel probability measure on $\mathrm{Z}$, denoted by $\nu_{\infty}$, such that 
	\begin{equation} \label{rn-cont1.5}
	\nu_{\lambda_j} \rightharpoonup \nu_{\infty} \mbox{~in~the~weak-$\star$~topology}.
	\end{equation}

We claim that $\nu_{\infty}$ is an invariant Borel probability measure of $\big\{\Psi^t_{\lambda_0}\big\}$ on $\mathrm{Z}$. By \cite[Theorem 6.8]{Wa82}, it is equivalent to show that for any continuous function $f:\mathrm{Z} \to \R$, the following holds
	\[
	\int_{\mathrm{Z}} f(\vartheta,\widetilde{\textbf{v}}) \rd \nu_{\infty}= \int_{\mathrm{Z}} f\big( \Psi^t_{\lambda_0}(\vartheta,\widetilde{\textbf{v}})\big) \rd \nu_{\infty} \qquad \mbox{for~all~} t \in \R.
	\]
Without loss of generality, it suffices to verify the case $t>0$. To this end, for any fixed $t \in \R_+$ and any sequence $\lambda_j \to \lambda_0$, by Lemma \ref{th-la} and Lemma \ref{lm-theta} \romannumeral3), we obtain
	\[
	\theta_{\lambda_j}(t,\mathtt{\Theta};\widetilde{\bf{v}}) \to \theta_{\lambda_0}(t,\mathtt{\Theta};\widetilde{\bf{v}}) \qquad \mbox{uniformly~for~all~}\mathtt{\Theta} \in \R \mbox{~and~} \widetilde{\bf{v}} \in \mathrm{E}(\textbf{v}).
	\]
Combining this with \x{skew}, we obtain 
	\begin{equation} \label{rn-cont2}
	f\big(\Psi^t_{\lambda_j} (\vartheta,\widetilde{\textbf{v}})\big) \to  f\big(\Psi^t_{\lambda_0} (\vartheta,\widetilde{\textbf{v}})\big) \qquad \mbox{uniformly~for~all~}(\vartheta,\widetilde{\textbf{v}})\in \mathrm{Z}.
	\end{equation}
Since $\nu_{\lambda_j}$ is invariant under $\big\{\Psi^t_{\lambda_j}\big\}$, we then have
	\[
	\int_{\mathrm{Z}} f(\vartheta,\widetilde{\textbf{v}}) \rd \nu_{\lambda_j}= \int_{\mathrm{Z}} f\big( \Psi^t_{\lambda_j}(\vartheta,\widetilde{\textbf{v}})\big) \rd \nu_{\infty}.
	\]
It follows from \x{rn-cont1.5} and \x{rn-cont2} that
	\begin{equation*}
	\begin{split}
		\int_{\mathrm{Z}}f(\vartheta,\widetilde{\textbf{v}})\mathrm{d}\nu_{\lambda_j}
		&=\int_{\mathrm{Z}}\big(f\circ\Psi_{\lambda_j}^t-f\circ\Psi_{\lambda_0}^t\big)(\vartheta,\widetilde{\textbf{v}})\mathrm{d}\nu_{\lambda_j}+\int_{\mathrm{Z}}f\big(\Psi_{\lambda_0}^t(\vartheta,\widetilde{\textbf{v}})\big)\mathrm{d}\nu_{\lambda_j} \\
		&\rightarrow \int_{\mathrm{Z}}f\big(\Phi_{\lambda_0}^t(\vartheta,\widetilde{\textbf{v}})\big)\mathrm{d}\nu_{\infty} \qquad \mbox{as~} j \to +\infty.
	\end{split}
	\end{equation*}
By \x{rn-cont1.5} again, we have 
	\[
	\int_{\mathrm{Z}}f(\vartheta,\widetilde{\textbf{v}})\mathrm{d}\nu_{\lambda_j}\rightarrow \int_{\mathrm{Z}}f(\vartheta,\widetilde{\textbf{v}})\mathrm{d}\nu_{\infty} \qquad \mbox{as~} j \to +\infty.
	\]
The proof of the claim is complete.

By Remark \ref{re-rn} \romannumeral2) and this claim, we have
	\[	\rho(\lambda_0,\textbf{v})=\int_\mathrm{Z}f_{\lambda_0}(\vartheta,\widetilde{\textbf{v}})\mathrm{d}\nu_{\infty}.
	\] 
Combining this with \x{rn-cont1.1}, we obtain
	\begin{equation*}
	\begin{split}
		& \left| \rho(\lambda_j,{\bf{v}})-\rho(\lambda_0,{\bf{v}})\right| 
		=\left|\int_\mathrm{Z}f_{\lambda_j}(\vartheta,\widetilde{\textbf{v}})\mathrm{d}\nu_{\lambda_j} -\int_\mathrm{Z}f_{\lambda_0}(\vartheta,\widetilde{\textbf{v}})\mathrm{d}\nu_{\infty}\right|\\
		&\qquad \leq \left| \int_{\mathrm{Z}}\big(f_{\lambda_j}-f_{\lambda_0}\big)(\vartheta,\widetilde{\textbf{v}})\mathrm{d}\nu_j\right| +\left|\int_\mathrm{Z}f_{\lambda_0}(\vartheta,\widetilde{\textbf{v}})\mathrm{d}\nu_{\lambda_j}-\int_\mathrm{Z}f_{\lambda_0}(\vartheta,\widetilde{\textbf{v}})\mathrm{d}\nu_{\infty} \right| \\
		& \qquad \leq|\lambda_j-\lambda_0|\|\widetilde{w}\|_{\infty}+\left|\int_\mathrm{Z}f_{\lambda_0}(\vartheta,\widetilde{\textbf{v}})\mathrm{d}\nu_{\lambda_j}-\int_\mathrm{Z}f_{\lambda_0}(\vartheta,\widetilde{\textbf{v}})\mathrm{d}\nu_{\infty} \right| \rightarrow 0,
	\end{split}
	\end{equation*}
where \x{df-f}, Lemma \ref{lm-f} and \x{rn-cont1.5} are used. We have the desired result \x{rn-cont1}.\end{proof} 

Let $\phi(x):=\phi(x,\lambda;\bf{v})$ be a non-trivial solution of \x{tau-lam} with $\widetilde{\bf v}=\bf{v}$. Due to \x{v} and Lemma \ref{lm-1/f} \romannumeral1), $\phi(x)$ only has non-degenerate zeros, that is, if $\phi(x)=0$ then $\phi'(x)\neq 0$. For any fixed $\lambda$ and $\bf{v}$, we define
	\[
	N_{\phi}(x,\lambda;{\bf{v}}):=\sharp\{s\in [0,x):\phi(s,\lambda;{\bf{v}})=0\} \qquad \mbox{for~}x \in \R_+.
	\]
 In addition to the ergodic representation of rotation numbers given in Remark \ref{re-rn} \romannumeral2), we introduce the following geometric representation of rotation numbers, which will be used in the next section.

\begin{lem} \label{geo-rn}
	Let $\lambda$ and ${\bf{v}}$ be fixed. Then we have
		\begin{equation*}
		\lim\limits_{x\rightarrow+\infty}\frac{\pi N_\phi(x,\lambda;{\bf{v}})}{x}=\rho(\lambda,{\bf{v}}),
		\end{equation*}
	where the limit is independent of the choice of non-trivial solutions\/ $\phi(x)$.
\end{lem}

\begin{proof}
By \x{pufer}, the condition $\phi(x)=0$ is equivalent to $\theta(x)=k\pi$ for some $k \in \Z$. Recalling (\ref{ftheta}), we obtain
	\[
	\left.\frac{\mathrm{d}\theta(x)}{\mathrm{d}x}\right|_{\theta(x)=k\pi}=\frac{1}{p(x)}>0.
	\]
This implies that 
	\[
	\big|(\theta_\lambda(x,\mathtt{\Theta};{\bf{v}})-\theta_\lambda(0,\mathtt{\Theta};{\bf{v}}))-\pi \cdot N_\phi(x,\lambda;{\bf{v}})\big|<2\pi \qq \mbox{for~} x \in \R_+.
	\]
Combining this with Theorem \ref{thm-rn}, we complete the proof.
\end{proof}

Moreover, we have
\begin{lem}
	Let $\bf{v}$ be fixed. Then $\rho(\lambda,{\bf{v}})\geq0$ and it is non-decreasing with respect to\/ $\lambda$.
\end{lem}
\begin{proof}
The non-negativity of $\rho(\lambda,{\bf{v}})$ follows from the fact that $N(x,\lambda;{\bf{v}})\geq 0$ for all $\lambda\in\mathbb{R}$. By Lemma \ref{th-la} and Theorem \ref{thm-rn}, we know that  $\rho(\lambda,{\bf{v}})$ is non-decreasing. The proof is complete. \end{proof}

\begin{rem} 
	{\rm
		We can also use Sturm's comparison theorem to conclude that $\rho(\lambda,{\bf{v}})$ is non-decreasing.
	}
\end{rem}

\section{Gap labelling}\label{sec-gl}

\setcounter{equation}{0}

In this section, we focus on the proof of the gap labelling thereom for $L_{\bf{v}}$ in terms of rotation numbers.

\begin{proof}[Proof of Theorem \ref{thm-gl}.]
For $\lambda\in\mathbb{R}\setminus\sigma(L_{\bf{v}})$ and $\widetilde{\bf{v}}\in\mathrm{E}(\textbf{v})$, the functions $\phi_{\pm}(x,\lambda;\widetilde{\bf{v}})$ can be chosen to be real-valued and normalized so that $\mathrm{W}(\phi_{+}(\cdot,\lambda;\widetilde{\bf{v}}),\phi_{-}(\cdot,\lambda;\widetilde{\bf{v}});\widetilde{\bf{v}})=1$. By \x{green} and Lemma \ref{spe-inv}, we have
	\begin{equation} \label{thm-gl2}
	G(x,x,\lambda;\widetilde{\bf{v}})=\phi_{+}(x,\lambda;\widetilde{\bf{v}})\phi_{-}(x,\lambda;\widetilde{\bf{v}})\qquad \mbox{for~} \lambda\in\mathbb{R}\setminus\sigma(L_{\bf{v}}).
	\end{equation}
Then $G(x,x,\lambda;\widetilde{\bf{v}})=0$ if and only if $\phi_{+}(x,\lambda;\widetilde{\bf{v}})=0$ or $\phi_{-}(x,\lambda;\widetilde{\bf{v}})=0$. Thus we have 
	\[
	\left.\frac{\mathrm{d}}{\mathrm{d}x}G(x,x,\lambda;\widetilde{\bf{v}})\right|_{G(x,x,\lambda;\widetilde{\bf{v}})=0}=\pm \frac{1}{\widetilde{p}(x)} \neq 0.
	\]
Based on the above discussion and Lemma \ref{gg'ap}, we know that $G(x,x,\lambda;{\bf{v}})$ satisfies all the requirements of Lemma \ref{lm-fzero}. Then the following holds
	\[
	\lim\limits_{x\rightarrow +\infty}\frac{\pi\sharp\{s\in [0,x):G(s,s,\lambda;{\bf{v}})=0\}}{x}\in \mathcal{M}_{\bf{v}}.
	\]
By \x{thm-gl2} again, it is obvious that
	\[
	\sharp\{s\in [0,x):G(s,s,\lambda;{\bf{v}})=0\}= 	N_{\phi_+}(x,\lambda;{\bf{v}})+ N_{\phi_-}(x,\lambda;{\bf{v}}).
	\] 
Combining this with Lemma \ref{geo-rn}, we have 
	\[
	2\rho(\lambda,{\bf{v}})\in \mathcal{M}_{\bf{v}}\qquad \mbox{for~} \lambda\in\mathbb{R}\setminus\sigma(L_{\bf{v}}).
	\]
By Lemma \ref{rn-cont}, we know that $\rho(\lambda,\bf{v})$ is a constant in any open interval $\mathrm{J}$ of $\mathbb{R}\setminus \sigma(L_{\bf{v}})$.
\end{proof}

\begin{rem} 
	{\rm
		Similarly as in {\rm Remark \ref{re-rn} \romannumeral1)}, if we consider the case $\widetilde{\bf{v}}\in\mathrm{E}(\bf{v})$, we have
			\[
			2\rho(\lambda,\widetilde{\bf{v}})\in \mathcal{M}_{\bf{v}}\qquad \mbox{for~} \lambda\in \mathrm{J},
			\]
		where $\mathrm{J}$ is any open interval of $\mathbb{R}\setminus \sigma(L_{\widetilde{\bf{v}}})$.
	}
\end{rem}

\section*{Acknowledgments}

Y.\ W.\ and Z.\ Z.\ are indebted to the Faculty of Mathematics at the University of Vienna for its hospitality during the winter of 2025/26 and 2024/25 respectively, where some part of this work was done.
Y.\ W.\ and Z.\ Z.\ are supported in part by the National Natural Science Foundation of China (Grants No. 12271509, 12090010, 12090014).

\end{document}